%% file: main.tex
\documentclass[11pt]{article}

\usepackage{wrapfig}
\usepackage{sectsty}
\sectionfont{\large}
\subsectionfont{\normalsize}
\subsubsectionfont{\small}
\usepackage{listings}
\usepackage{caption}
\usepackage{subcaption}

\usepackage{amsfonts}
\usepackage{microtype}
\usepackage{graphicx}
\usepackage{booktabs}
\usepackage{multicol}
\usepackage{algorithm}
\usepackage{algorithmic}
\usepackage{enumitem}

\usepackage{setspace}
\usepackage{natbib}

\usepackage{amsmath}
\usepackage{amssymb}
\usepackage{mathtools}
\usepackage{amsthm}

\usepackage{amsfonts}
\usepackage{amsmath}
\makeatletter
\newcommand*{\da@rightarrow}{\mathchar"0\hexnumber@\symAMSa 4B }
\newcommand*{\da@leftarrow}{\mathchar"0\hexnumber@\symAMSa 4C }
\newcommand*{\xdashrightarrow}[2][]{%
  \mathrel{%
    \mathpalette{\da@xarrow{#1}{#2}{}\da@rightarrow{\,}{}}{}%
  }%
}
\newcommand{\xdashleftarrow}[2][]{%
  \mathrel{%
    \mathpalette{\da@xarrow{#1}{#2}\da@leftarrow{}{}{\,}}{}%
  }%
}
\newcommand*{\da@xarrow}[7]{%
  \sbox0{$\ifx#7\scriptstyle\scriptscriptstyle\else\scriptstyle\fi#5#1#6\m@th$}%
  \sbox2{$\ifx#7\scriptstyle\scriptscriptstyle\else\scriptstyle\fi#5#2#6\m@th$}%
  \sbox4{$#7\dabar@\m@th$}%
  \dimen@=\wd0 %
  \ifdim\wd2 >\dimen@
    \dimen@=\wd2 %
  \fi
  \count@=2 %
  \def\da@bars{\dabar@\dabar@}%
  \@whiledim\count@\wd4<\dimen@\do{%
    \advance\count@\@ne
    \expandafter\def\expandafter\da@bars\expandafter{%
      \da@bars
      \dabar@ 
    }%
  }%
  \mathrel{#3}%
  \mathrel{%
    \mathop{\da@bars}\limits
    \ifx\\#1\\%
    \else
      _{\copy0}%
    \fi
    \ifx\\#2\\%
    \else
      ^{\copy2}%
    \fi
  }%
  \mathrel{#4}%
}
\makeatother

\usepackage[colorlinks,
            linkcolor=red,
            anchorcolor=blue,
            citecolor=blue
            ]{hyperref}
\usepackage[capitalize,noabbrev]{cleveref}

\theoremstyle{plain}
\newtheorem{theorem}{Theorem}[section]
\newtheorem{proposition}[theorem]{Proposition}
\newtheorem{lemma}[theorem]{Lemma}

\newtheorem{example}[theorem]{Example}

\theoremstyle{definition}

\newtheorem{assumption}[theorem]{Assumption}

\newtheorem{remark}[theorem]{Remark}

\usepackage{smile}
\usepackage{makecell}
\input{math_commands.tex}

\newcommand*{\T}{\mathsf{T}}

\DeclareMathOperator\diam{diam}
\newcommand{\interior}[1]{%
  {\kern0pt#1}^{\mathrm{o}}%
}

\newcommand\extrafootertext[1]{%
    \bgroup
    \renewcommand\thefootnote{\fnsymbol{footnote}}%
    \renewcommand\thempfootnote{\fnsymbol{mpfootnote}}%
    \footnotetext[0]{#1}%
    \egroup
}
\usepackage[textsize=tiny]{todonotes}
\usepackage{fullpage}

\usepackage{float}
\labelwidth\leftmargini\advance\labelwidth-\labelsep \labelsep 5pt

\def\@listi{\leftmargin\leftmargini}
\def\@listii{\leftmargin\leftmarginii
   \labelwidth\leftmarginii\advance\labelwidth-\labelsep
   \topsep 2pt plus 1pt minus 0.5pt
   \parsep 1pt plus 0.5pt minus 0.5pt
   \itemsep \parsep}
\def\@listiii{\leftmargin\leftmarginiii
    \labelwidth\leftmarginiii\advance\labelwidth-\labelsep
    \topsep 1pt plus 0.5pt minus 0.5pt
    \parsep \z@ \partopsep 0.5pt plus 0pt minus 0.5pt
    \itemsep \topsep}
\def\@listiv{\leftmargin\leftmarginiv
     \labelwidth\leftmarginiv\advance\labelwidth-\labelsep}
\def\@listv{\leftmargin\leftmarginv
     \labelwidth\leftmarginv\advance\labelwidth-\labelsep}
\def\@listvi{\leftmargin\leftmarginvi
     \labelwidth\leftmarginvi\advance\labelwidth-\labelsep}
     
\usepackage{authblk}

\title{\textbf{\Large{
Achieving Hierarchy-Free Approximation for Bilevel Programs With Equilibrium Constraints}}}

\author[1]{\small Jiayang Li}
\author[1]{\small Jing Yu}
\author[2]{\small Boyi Liu}
\author[2]{\small Zhaoran Wang}
\author[1]{\small Yu (Marco) Nie\footnote{Corresponding author; \texttt{y-nie@northwestern.edu}.}}

\affil[1]{\small \textit{Department of Civil and Environmental Engineering, Northwestern University}}

\affil[2]{\small \textit{Department of Industrial Engineering and Management Science, Northwestern University}}

\date{}

\begin{document}

\maketitle

\begin{abstract}
In this paper, we develop an approximation scheme for solving bilevel programs with equilibrium constraints, which are generally difficult to solve. Among other things, calculating the first-order derivative in such a problem requires differentiation across the hierarchy, which is computationally intensive, if not prohibitive.  To bypass the hierarchy, we propose to bound such bilevel programs, equivalent to multiple-followers Stackelberg games, with two new hierarchy-free problems: a $T$-step Cournot game and a $T$-step monopoly model. Since they are standard equilibrium or optimization problems, both can be efficiently solved via first-order methods. Importantly, we show that the bounds provided by these problems --- the upper bound by the $T$-step Cournot game and the lower bound by the $T$-step monopoly model --- can be made arbitrarily tight by increasing the step parameter $T$ for a wide range of problems. We prove that a small $T$ usually suffices under appropriate conditions to reach an approximation acceptable for most practical purposes. Eventually, the analytical insights are highlighted through numerical examples.
\end{abstract}

\section{Introduction}

Many bilevel optimization problems arising from real-world applications can be cast as a mathematical program whose feasible region is defined by an equilibrium problem \citep{luo1996mathematical, outrata1998nonsmooth}. A typical example is a Stackelberg game concerning a leader who aims to induce a desirable outcome in an economic or social system comprised of many self-interested followers, who can be seen as playing a non-cooperative game that converges to a Nash equilibrium \citep{dafermos1973toll, requate1993pollution, marcotte1992efficient, labbe1998bilevel, ehtamo2002recent}. More recently, motivated by such applications, developing efficient algorithms for solving bilevel programs with equilibrium constraints has also emerged as an essential topic in machine learning \citep{mguni2019coordinating, zheng2020ai, 2022liuinducing, maheshwari2022inducing}. 

In the optimization literature, a  bilevel program with equilibrium constraints is often written as \citep{luo1996mathematical}
\begin{equation}
\begin{split}
        \min_{\vx \in \sX, \ \vy^* \in \sY}~~&l(\vx, \vy^*),  \\
        \text{s.t.}~~&\langle f(\vx, \vy^*), \vy - \vy^* \rangle \geq 0, \quad \forall \vy \in \sY,
        \label{eq:bilevel}
\end{split}
\end{equation}
where $\sX \subseteq \sR^{m}$ and $\sY \subseteq \sR^{n}$ are two convex set; $l: \sX \times \sY \to \mathbb R$ and $f: \sX \times \sY \to \sR^{n}$ are two continuously differentiable functions. The lower-level problem in Problem \eqref{eq:bilevel} is a variational inequality (VI) problem, which is a general formulation for many equilibrium problems \citep{scutari2010convex, nagurney2013network, parise2019variational}. 
Problem \eqref{eq:bilevel} is well known for its intractability. Indeed,  it is NP-hard even when the upper-level objective function is linear, and the lower-level VI can be reduced to a linear program (LP) \citep{ben1990computational}, leading to a so-called bilevel linear program  \citep{candler1982linear,bard1982explicit}.

When the lower level is not an LP, Problem \eqref{eq:bilevel}  is usually solved via first-order methods that strive to find good local solutions \citep{colson2007overview}. Classical algorithms in this category include basic gradient descent \citep{friesz1990sensitivity}, steepest descent built on quadratic approximation \citep{luo1996mathematical}, and the penalty method \citep{aiyoshi1984solution}.
Applying a gradient descent method requires differentiation through the lower-level equilibrium problem, which is a challenging computational task. In the literature, it is often accomplished via \textit{implicit differentiation} (ID), which requires inverting a matrix whose size scales quadratically with the dimension of the lower-level VI problem \citep{tobin1986sensitivity, dafermos1988sensitivity, parise2019variational}. In large-scale problems, even storing such a matrix may be impractical, let alone inverting them.

The recent advance in machine learning (ML) has inspired a new class of algorithms for solving bilevel programs based on \textit{automatic differentiation} (AD) \citep{griewank1989automatic}. In AD-based methods, the gradient of a bilevel program is computed in two phases \citep{franceschi2017forward, franceschi2018bilevel}. In the first phase, the lower-level problem is first solved, while the computation process, along with intermediate results, is stored in a computational graph. In the second phase, the gradient of the lower-level solution is evaluated by \emph{unrolling} the computational graph. In the literature, AD-based methods were originally proposed for ML applications, e.g., hyperparameter optimization \citep{maclaurin2015gradient} and neural architecture search \citep{liu2018darts}, which can usually be formulated as a bilevel program whose lower level is an unconstrained optimization problem. More recently, they were also extended to handle those with equilibrium constraints \citep{li2020end}. Although AD-based methods bypass implicit differentiation, they may run into another challenge: since the computational graph grows with the number of iterations required to solve the lower-level equilibrium problem, it may become too deep to unroll efficiently even with AD \citep{li2022differentiable} when solving the lower-level problem requires too many iterations.

In a nutshell, finding the gradient for Problem \eqref{eq:bilevel} remains a potential obstacle to large-scale applications, whether ID or AD is used. These difficulties have motivated many work to accelerate ID \citep{hong2020two, chen2021single, liao2018reviving, grazzi2020iteration, vicol2021implicit, fung2022jfb, 2022liuinducing} or approximate AD \citep{luketina2016scalable, metz2016unrolled, finn2017model, liu2018darts, shaban2019truncated, ablin2020super, yang2021provably, li2022fully}. But fundamentally, the difficulty is inherent in the \textit{hierarchy} of the problem, or the fact that to obtain the gradient, one is obligated to solve and differentiate through the lower-level problem, which is computationally demanding in many cases. Our work is prompted by the following question: \textit{is it possible to free the process from that obligation, or to ``bypass the hierarchy"?}

We believe a hierarchy-free method is possible. Our inspiration comes from the duopoly model in economics, which concerns two firms, A and B, selling a homogeneous product in the same market. The duopoly can be organized in three ways \citep{shapiro1989theories}. (1) \textit{Stackelberg duopoly}: Firm A sets its output first, according to which Firm B makes the decision, giving rise to a typical bilevel program. (2) \textit{Cournot duopoly}: Firms A and B simultaneously optimize their own output, which results in a Nash equilibrium problem. (3) \textit{Monopoly}: Firm A becomes the only producer by taking over Firm B's business and setting the total output for both. In economics, it is well known that Firm A's optimal profit in the Stackelberg duopoly is lower than that in a monopoly but higher than that in a Cournot duopoly.

As Problem \eqref{eq:bilevel} can be 
interpreted as a Stackelberg game in which the leader and the followers control the upper- and lower-level decision variables, respectively, we reason that it may be bounded in a similar way as the Stackelberg duopoly is bounded. Specifically, instead of directly solving Problem \eqref{eq:bilevel}, we may first solve the corresponding ``Cournot game" and ``monopoly model" --- both of which are single-level problems --- to obtain lower and upper bounds. If the two bounds are close enough, we may accept the feasible one as an approximate solution.
The caveat, of course, is that the natural gap between these two models may be unacceptably large for practical purposes. Thus, the focus of this investigation is to narrow down this gap.

\textbf{Our contribution.}
In this paper, we view Problem \eqref{eq:bilevel} as a Stackelberg game and develop a new Cournot game and a new monopoly model that can provide arbitrarily tight upper and lower bounds for Problem \eqref{eq:bilevel}. The development of both models assumes the lower-level equilibrium state is the outcome of a dynamical process through which the followers improve their decisions \textit{step-by-step} towards optimality \citep{weibull1997evolutionary}. The two proposed models are defined as follows: in a $T$-step Cournot game, the leader and followers make decisions simultaneously, but the leader \textit{anticipates} the followers' decisions by $T$ steps, while in a $T$-step monopoly model, the leader has full \textit{control} over the followers but allows them to move on their dynamical process toward equilibrium by $T$ steps after the leader first dictates their decision.

Our contributions are threefold. (1) We show that both models can be efficiently solved via first-order methods, and the computation cost in each iteration grows linearly with $T$. (2) We prove that under appropriate assumptions and by choosing a suitable $T$, the gap between the upper and lower bounds provided by the solutions to the two models becomes arbitrarily tight; for most practical purposes, a small $T$  suffices to provide a high-quality approximate solution to Problem \eqref{eq:bilevel}. (3) We demonstrate the applications of the proposed approximation scheme in a range of real-world problems.

\textbf{Organization.} In Section \ref{sec:application}, we highlight a few real-world applications that motivate the present study. In Section \ref{sec:challenge}, we discuss the difficulties in solving Problem \eqref{eq:bilevel}, along with a review of how they are addressed in previous research. Section \ref{sec:duopoly} motivates the proposed hierarchy-free scheme by drawing an analogy between Problem \eqref{eq:bilevel} and the classical Stackelberg duopoly model. Section \ref{sec:main-results} lays the foundation for the scheme: the formulations of the $T$-step Cournot game and the $T$-step monopoly model before presenting the solution algorithm and discussing the analytical properties of the scheme. Finally, Section \ref{sec:experiment} presents numerical results, and Section \ref{sec:conclusion} concludes the paper.

\textbf{Notation.} We use $\sR$, $\sR_{+}$, and $\sN$ to denote the set of real numbers, non-negative real numbers, and non-negative integers. 
The inner product of two vectors $\va, \vb \in \sR^n$ is written as $\langle \va, \vb \rangle = a^{\T} b$.
For a matrix $\mA \in \sR^{n \times m}$, we denote $\|\mA\|_2$ as its matrix norm induced by the 2-norm for vectors.  For a closed set $\sA$, we denote $\diam{(\sA)} = \max_{\va, \va' \in \sA} \|\va - \va'\|$ as its diameter. For a finite set $\sA$, we write $|\sA|$ as the number of elements in $\sA$ and $\Delta(\sA) = \{\vp \in \sR_+^{|\sA|}: \vone^{\T} \vp = 1\}$.

\section{Background}
\label{sec:application}

We will first discuss how VI provides a unified formulation for many equilibrium problems (Section \ref{sec:vi}) and then introduce a few real-world applications that motivate bilevel programs with equilibrium constraints (Section \ref{sec:stackelberg})

\subsection{Variational Inequalities}
\label{sec:vi}

VI provides a unified formulation for both atomic and nonatomic games, classified according to whether the set of agents is endowed with an atomic or a nonatomic measure  \citep{nash1951non, schmeidler1973equilibrium}. Simply put, each agent's decision can affect the outcome of an atomic game, but the outcome of a nonatomic game solely depends on the aggregate behavior of the agents.

\begin{example}[Atomic game]
\label{eg:nash}
Consider a game played by $n$ atomic agents. Suppose that each agent $i$ aims to select a strategy $\vz_i \in \sZ_i \subseteq \sR^{m_i}$ to minimize its cost, which a determined by a continuously differentiable function $u_i: \sZ \to \mathbb R$ where $\sZ = \prod_{i} \sZ_i$. Formally, a joint strategy $z^* \in \sZ$ is a Nash equilibrium if 
$u_i(\vz_i^*, \vz_{-i}^*) = \min_{\vz_i \in \sZ_i} u_i(\vz_i, \vz_{-i}^*)$ ($i = 1, \ldots, k$).
Denote $v(\vz) = (v_i(\vz))_{i = 1}^k$ be a function with $v_i(\vz) = \nabla_{\vz_i}u(\vz)$. 
Suppose that $\sZ_i$ is convex and closed, then any Nash equilibrium $\vz^* = (\vz_i^*)_{i = 1}^k \in \sZ$ is also a solution to the following VI \citep{scutari2010convex}
\begin{equation}
    \langle v(\vz^*), \vz - \vz^* \rangle \geq 0, \quad \forall \vz \in \sZ.
   \label{eq:nash-vi}
\end{equation}
Meanwhile, the reverse also holds if each $v_i$ is convex in $\vz_i$.
\end{example}

\begin{example}[Nonatomic game]
\label{eg:wardrop}
Consider a game played by $n$ classes of nonatomic agents. Let $\sA_i$ be the discrete action set for agents in class $i$. Let $\vp_{i} = (\evp_{ia})_{a \in \sA_i} \in \Delta(\sA_i)$ be the proportion of agents in class $i$ selecting each action $a \in \sA_i$. Suppose that each agent in class $i$ aims to select an action $a$ to minimize the cost determined by a continuous function $c_{ia}: \Delta(\sA) \to \mathbb R$ where $\Delta(\sA) = \prod_i \Delta(\sA_i)$. Formally, a mass distribution $\vp^* = (\vp_i^*)_{i = 1}^k \in \Delta(\sA)$ is a Nash equilibrium (also known as a Wardrop equilibrium) if $c_{ia'}(\vp^*) = \min_{a \in \sA_i} c_{ia}(\vp^*)$ for all $a' \in \sA_i$ satisfying $\evq_{ia'} > 0$ ($i = 1, \ldots, k$).
Letting $c_{i}(\vp) = (c_{ia}(\vp))_{a \in \sA_i}$ and $c(\vp) = (c_i(\vp))_{i = 1}^k$, then $\vp^*$ is a Nash equilibrium if and only if \citep{bernhard2011ess}
\begin{equation}
    \langle c(\vp^*), \vp - \vp^* \rangle \geq 0, \quad \forall \vp \in \Delta(\sA).
   \label{eq:wardrop-vi}
\end{equation}
\end{example}

As %
most equilibrium problems can be cast as VI, Problem \eqref{eq:bilevel} provides a standard formulation for bilevel programs with equilibrium constraints \citep{luo1996mathematical}.

\subsection{Bilevel Programs with Equilibrium Constraints}
\label{sec:stackelberg}

The study of bilevel programs can be traced back to the Stackelberg duopoly model. 

\begin{example}[Stackelberg duopoly]
\label{eg:duopoly}
Consider two firms, A and B, selling a homogeneous product. Let their outputs be denoted, respectively, as $x$ and $y$, and suppose that Firm A chooses $x$ first, and then Firm B chooses $y$ subsequently. Let the inverse demand (i.e., price) for the product be $p = 1 - x - y$. Then the profits for firms A and B are given as $l(x, y) = x (1 -  x -  y)$ and $g(x, y) = y (1 - x - y)$. The optimal decisions $x^*$ and $y^*$ of the two firms are the solution to the following bilevel program
\begin{equation}
    x^* = \argmax_{x \geq 0}~l(x, y^*), \quad \text{s.t.}~y^* = \argmax_{y \geq 0}~g(x, y).
    \label{eq:stackelberg-duopoly}
\end{equation}
The optimal solution is $x^* = 1/2$ and $y^* = 1 / 4$, with Firm A making an optimal profit of  $1/8$.
\end{example}

The earliest study on bilevel programs with equilibrium constraints was motivated by Stackelberg congestion games (SCGs), which concern a leader (usually a traffic planner) who aims to induce a desirable equilibrium state in a congestion game \citep{wardrop1952road, roughgarden2002bad} played by many self-interested followers (travelers). The network design problem \citep{leblanc1975algorithm, li2012global} and the congestion pricing problem \citep{lawphongpanich2004mpec, li2021traffic} are two classic examples.  More recently, the study of SCGs has been influenced by the introduction and constant evolution of connected and automated vehicle (CAV) technologies \citep{mahmassani201650th}, leading to such applications as the design of dedicated CAV facilities \citep{chen2016optimal,chen2017optimal,bahrami2020optimal} and the control of CAVs within such facilities \citep{levin2016cell,zhang2018mitigating}.

The question of inducing a desirable outcome in non-cooperative games can be traced back to the work of \citet{pigou1920economics} on welfare economics. Bilevel programming has long been recognized as the standard approach to such inquiries in operations research, and more recently in the ML community \citep{mguni2019coordinating, zheng2020ai, 2022liuinducing, maheshwari2022inducing}. Our algorithms are focused on the applications pertinent to this question.

We hope to clarify that not all bilevel programs (and Stackelberg games) are amenable to our algorithms. The \textit{first} class is the Stackelberg games played by one leader and one follower in which the action sets of both are finite. Such problems can be reformulated as a linear program; examples include the generalized principal-agent problem \citep{myerson1982optimal} and the Stackelberg security game \citep{sinha2018stackelberg}. The \textit{second} class is a bilevel program constrained by an LP \citep{bracken1973mathematical}, which is equivalent to an NP-hard mixed-integer program \citep{ben1990computational} that cannot be effectively solved via first-order methods.

\section{Challenges}
\label{sec:challenge}

In this section, we will discuss the difficulties in computing the first-order gradient of  Problem \eqref{eq:bilevel}, which reads
\begin{equation}
    \frac{\partial l(\vx, \vy^*)}{\partial \vx} = \nabla_{\vx} l(\vx, \vy^*) + \frac{\partial \vy^*}{\partial \vx} \cdot \nabla_{\vy} l(\vx, \vy^*).
    \label{eq:gradient}
\end{equation}
To obtain this gradient, we need to \textit{solve} and \textit{differentiate through} the lower-level VI problem. 

Given any $\vx \in \sX$, we denote the solution set to the lower-level VI problem in Problem \eqref{eq:bilevel} as $\sY^*(\vx)$. We first give the following proposition for characterizing $\sY^*(\vx)$.

\begin{proposition}[\citet{hartman1966some}]
\label{prop:fixed-point}
Suppose that $\sY$ is closed. Let $h: \sX \times \sY \to \sY$ be a function that satisfies
\begin{equation}
    h(\vx, \vy) = \argmin_{\vy' \in \sY}~\|\vy' - r \cdot f(\vx, \vy) \|_2^2.
    \label{eq:bregman-projection}
\end{equation}
Then given any $\vx \in \sX$, we have $\vy^* \in \sY^*(\vx)$ if and only if $\vy^*$ is a fixed point of $h(\vx, \cdot)$, i.e., $\vy^* = h(\vx, \vy^*)$.
\end{proposition}

To solve $\sY^*(\vx)$, Proposition \ref{prop:fixed-point} has inspired a general class of algorithms, commonly known as the
projection method \citet{dafermos1983iterative, pang1982iterative, marcotte1995convergence}, which iteratively project $\vy^t$ to $\vy^{t + 1} = h(\vx; \vy^t)$, starting from some $\vy^0 \in \sY$, until a fixed point is found.

To differentiate through $\sY^*(\vx)$, however, the lower-level VI problem must admit a unique solution; otherwise, the Jacobian matrix $\partial \vy^*/\partial \vx$ is not well defined. To secure the uniqueness, one often needs to assume $f(\vx, \cdot)$ is strongly monotone \citep{mancino1972convex}.
Our work follows many previous studies \citep{ghadimi2018approximation, hong2020two, chen2021single, guo2021randomized, ji2021bilevel, 2022liuinducing} to adopt this assumption, but we will discuss how to relax it in the appendix. The following proposition characterizes the convergence rate of the projection method when $f(\vx, \cdot)$ is strongly monotone.

\begin{proposition}[\citet{nagurney2013network}]
\label{thm:convergence-strong}
Suppose that $f(\vx, \cdot)$ is $\gamma$-strongly monotone and $L$-Lipschitz continuous, then
\begin{equation}
    \|h(\vx,  \vy) - h(\vx, \vy')\|_2 \leq  \eta \cdot \|\vy - \vy'\|_2
\end{equation}
for all $\vy, \vy' \in \sY$, where $\eta = (1 - 2\gamma r /\sigma + r^2L^2 / \sigma^2)^{1/2}$. Hence, starting from any $\vy^0 \in \sY$, then the sequence $\vy^{t + 1} = h^(\vx, \vy^t)$ converges to the unique point $\vy^* \in \sY^*(\vx)$ at a linear rate as long as the step size $r < 2\gamma / L^2$.
\end{proposition}

In the remainder of this section, we will discuss the calculation of $\partial \vy^* / \partial \vx$, which is the main obstacle behind implementing any first-order methods.

\subsection{Implicit Differentiation (ID) Methods}
\label{sec:id}

The first method to calculate $\partial \vy^* / \partial \vx$ is to implicitly differentiate through the fixed-point equation $\vy^* = h(\vx, \vy^*)$, which subsequently gives rise to the following proposition.

\begin{proposition}[\citet{dafermos1988sensitivity}]
\label{prop:implicit-differentiation}
If $h(\vx, \vy)$ is continuously differentiable and $f(\vx, \cdot)$ is strongly monotone, then the unique $\vy^* \in \sY^*(\vx)$ is continuously differentiable in $\vx$ with the Jacobian matrix satisfying
\begin{equation}
    \frac{\partial \vy^*}{\partial \vx} = \nabla_{\vx} h(\vx, \vy^*) \cdot (\mI - \nabla_{\vy} h(\vx, \vy^*))^{-1}.
    \label{eq:nabla-y-star}
\end{equation}
\end{proposition}
To calculate $\partial \vy^* / \partial \vx$ according to Equation 
\eqref{eq:nabla-y-star}, one first needs to obtain $\nabla_{\vx} h(\vx, \vy^*)$ and $\nabla_{\vy} h(\vx, \vy^*)$, that is, differentiating through a Euclidean projection problem, equivalent to a quadratic program (QP). One way to perform it is using the Python package \texttt{cvxpylayers} developed by \citet{agrawal2019differentiable}. The computational cost of implicit differentiation (ID) is high because it requires solving the lower-level VI problem and inverting a matrix that can be prohibitively large.

\textbf{\textit{Single-looped ID}.} To prevent repeatedly solving for $\vy^*$, one could update both $\vx$ and $\vy$ by one gradient-descent step at each iteration in a single loop. In such schemes, instead of calculating the exact upper-level gradient via Equations \eqref{eq:gradient} and \eqref{eq:nabla-y-star}, we replace $\vy^*$ therein by the current iteration. The scheme was initially proposed by \citet{hong2020two, chen2021single} and later extended to Problem \eqref{eq:bilevel} by \citet{2022liuinducing}. The single-loop scheme simplifies the overall structure of ID-based methods but does not bypass the difficulty of inverting large matrices. 

\textbf{\textit{Approximated ID}.} To prevent matrix inversion in \eqref{eq:nabla-y-star}, one can represent its inversion by the corresponding Neumann series and then truncate the series by keeping only its first few terms \citep{liao2018reviving, grazzi2020iteration, vicol2021implicit}. The Jacobian-free scheme proposed by \citet{fung2022jfb} can also be interpreted through Neumann-series truncation. The scheme significantly reduces the time complexity but requires storing $\nabla_{\vy} h(\vx, \vy^*)$.

There are other schemes designed to improve ID. For example, \citet{bertrand2020implicit} developed a matrix-free ID scheme for lasso-type problems; 
\citet{blondel2021efficient} developed a toolbox combining ID and AD benefits; \citet{sowconvergence} proposed a method that adopts a zeroth-order-like estimator to approximate the Jacobian matrix.

\subsection{Automatic Differentiation Methods}
\label{sec:ad}

The second method to compute $\partial \vy^* / \partial \vx$ is to unroll the computation process for solving the lower-level VI problem via AD \citep{franceschi2017forward, franceschi2018bilevel}. For example, if the projection method is adopted, the computation process can be written as
\begin{equation}
    \vy^{t} = h(\vx, \vy^{t - 1}), \quad t = 1, \ldots, T,
\end{equation}
where $T$ is a sufficiently large number such that the distance between $\vy^{T}$ and $\sY^*(\vx)$ is smaller than a tolerance value. The aforementioned \texttt{cvxpylayers} proposed by \citet{agrawal2019differentiable} package can be employed to wrap the computation process behind each $\vy^{t} = h(\vx, \vy^{t - 1})$ as a computational graph.

Since the computational graph grows with the number of iterations required to solve the lower-level equilibrium problem, it may become too deep to unroll efficiently even with AD, when solving the equilibrium problem requires too many iterations. Particularly for Problem \eqref{eq:bilevel}, $h(\vx, \vy^{t - 1})$ is equivalent to a constrained QP, which is more costly to store than in most ML applications, whose lower-level problem is typically unconstrained.

\textbf{\textit{Truncated AD}.} The difficulty in storing a large computational graph may be bypassed by truncated AD, which, by only unrolling the last portion of the graph, settles for an approximate gradient \citep{shaban2019truncated}.

\textbf{\textit{One-stage AD}.} Another approximation scheme is called one-stage AD, which updates $\vx$ and $\vy$  simultaneously in a single loop. Specifically, whenever $\vy$ in the lower level is updated by one step, one-stage AD unrolls it to obtain the gradient for updating $\vx$ in the upper level. The scheme has delivered satisfactory performance on many tasks \citep{luketina2016scalable, metz2016unrolled, finn2017model, liu2018darts, xu2019pc}. The method proposed by \citet{li2022fully} also shares a similar single-loop structure.

The performance of the above approximation schemes has been extensively tested \citep{franceschi2018bilevel,wu2018understanding}.  \citet{ablin2020super} discussed the efficiency of AD; \citet{ji2021bilevel} analyzed the convergence rate of AD and approximated ID; \citet{yang2021provably} and \citet{dagreou2022framework} studied variance reduction in AD-based methods. 
For other bilevel programming algorithms for ML applications, see, e.g.,\citet{pedregosa2016hyperparameter,lorraine2018stochastic,mackay2019self,bae2020delta, ji2021bilevel,grazzi2021convergence,zucchet2022beyond}. The reader may also consult \citet{liu2021investigating} for a survey.

\textbf{Our scheme.} A main difference between our scheme and the previous works is that we do not attempt to approximate the Jacobian matrix $\partial \vy^* / \partial \vx$ returned by exact ID or AD. Instead, we directly approximate the original bilevel program with two hierarchy-free models inspired by the classic economic competition theory; to the best of our knowledge, this angle is novel, even though the resulting algorithms do share some features with existing ID- and AD-based methods, as we shall see.

\section{Motivation}
\label{sec:duopoly}

In this section, we motivate the proposed scheme by discussing how the Stackelberg duopoly model (cf. Example \ref{eg:duopoly}) can be bounded by single-level models.

\subsection{Classic Models}

We first introduce the market structures of classic Cournot duopoly and monopoly models based on Example \ref{eg:duopoly}.

\textbf{Cournot duopoly.} Firm A loses its first-mover advantage and has to make decisions simultaneously with Firm B. The optimal decisions $\bar x$ and $\bar y$ of the two firms can be found by solving the following equilibrium problem
\begin{equation}
    \begin{cases}
        \displaystyle \bar x = \argmax_{x \geq 0}~l(x, \bar y), \\
        \displaystyle \bar y = \argmax_{y \geq 0}~g(\bar x, y).
    \end{cases}
    \label{eq:cournot-duopoly}
\end{equation}
The optimal solution is $\bar x = 1/3$ and $\bar y = 1/3$, with Firm A earning an optimal profit of $1/9$.

\textbf{Monopoly.}  Firm B is taken over by Firm A. To find the optimal production levels $\hat x$ and $\hat y$, we need to solve the following optimization problem
\begin{equation}
    \hat x, \hat y = \argmax_{x, y \geq 0}~l(x, y).
    \label{eq:monopoly-duopoly}
\end{equation}
The optimal solution is $\hat x = 1/2$ and $\hat y = 0$, and the optimal profit of Firm A rises to $1/4$.

\textbf{Comparision.} Intuitively, when the leader gives up the first-mover advantage, its market power is weakened; on the contrary, when the leader takes over Firm B's business, its market power is maximized. Indeed, comparing the optimal profits of Firm A, we have $1/4 > 1/8 > 1/9$, which indicates that a Stackelberg duopoly is bounded by a monopoly from above and by a Cournot duopoly from below.

\subsection{Our Models}
Our focus is to narrow the gap between the upper and lower bounds. To this end, we ``decompose" the best response of Firm B against Firm A's decision into a dynamical process. Since Firm B faces a constrained optimization problem, the projected gradient descent method charts a natural path to optimality. Specifically, given $x$ and a step size $r > 0$, the 
dynamical process through which Firm B finds its optimal decision can be described by a gradient-descent step, represented by
\begin{equation*}
    h(x, y) = \max\{y - r \cdot (1 - x - 2y), 0\}.
\end{equation*}
Starting with $h^{(0)}(x, y) := y$, thus, $h^{(t + 1)}(x, y)$ can be defined recursively as $h(x, h^{(t)}(x, y))$ ($t = 0, 1, \ldots$).

We are now ready to introduce the $T$-step Cournot duopoly model and the $T$-step monopoly model.

\textbf{$T$-step Cournot duopoly.}
Firm A and Firm B still choose $x$ and $y$ simultaneously, but Firm A bases its decision on the expectation that Firm B will update its decision $T$ times given $x$, i.e., from $y$ to $h^{(T)}(x, y)$. Therefore, the optimal decisions $\bar x$ and $\bar y$ can be found by solving 
\begin{equation}
    \begin{cases}
    \displaystyle \bar x = \argmax_{x \geq 0}~l(x, h^{(T)}(x, \bar y)), \\
    \displaystyle \bar y = \argmax_{y \geq 0}~g(\bar x, y).
    \end{cases}
    \label{eq:T-cournot-duopoly}
\end{equation}

\textbf{$T$-step monopoly.} Firm A imposes a decision $y$ on Firm B but allows Firm B to update its decision $T$ steps starting from $y$. To find the optimal decision $\hat x$ and $\hat y$ ($\hat y$ refers to the production of Firm B initially dictated by Firm A), we need to solve
\begin{equation}
    \hat x, \hat y = \argmax_{x, y \geq 0}~l(x, h^{(T)}(x, \bar y)).
    \label{eq:T-monopoly-duopoly}
\end{equation}
The question we set out to answer is the following: \textit{will Firm A's optimal profit in the two new models be closer to that in the Stackelberg duopoly}? Intuitively, the answer is yes, because the $T$-step Cournot duopoly interpolates ``no anticipation" and ``full anticipation" on Firm B's response, whereas the $T$-step duopoly interpolates ``full dictation" and  ``no dictation" on Firm B's decision. Hence, the market power of Firm A will rise and fall, respectively, in $T$-step monopoly and $T$-step Cournot, leading to an optimal profit closer to that in the Stackelberg duopoly.

\textbf{A key observation.} From the formulations, we find the two proposed models can be respectively viewed as a classic Cournot duopoly and a classic monopoly, with Firm A's objective function changed from $l(x, y)$ to $l^{(T)}(x, y) := l(x, h^{(T)}(x, y))$.
Based on this observation, we are now ready to approximate Problem \eqref{eq:bilevel} with the two models.

\section{Main Results}
\label{sec:main-results}

We view Problem \eqref{eq:bilevel} as a Stackelberg game and refer to the upper- and lower-level decision-makers therein as the leader and the followers, respectively. 
To extend the two models presented in Section \ref{sec:duopoly}, we first need a function to model the process through which the followers iteratively update their decisions towards equilibrium. As the lower level is a VI, the function $h(\vx, \vy)$ defined by Equation \eqref{eq:bregman-projection} is a natural choice. Similarly, starting with $h^{(0)}(\vx, \vy) := \vy$, we can recursively define $h^{(t + 1)}(\vx, \vy) = h(\vx, h^{(t)}(\vx, \vy))$ and $l^{(T)}(\vx, \vy) = l(\vx, h^{(T)}(\vx, \vy))$.

\textbf{$T$-step Cournot game.}
The leader and the followers choose $\vx \in \sX$ and $\vy \in \sY$ simultaneously, but the leader looks ahead $T$ steps along the followers' evolutionary path. The equilibrium state $(\bar \vx, \bar \vy)$ thus can be found by solving
\begin{equation}
    \begin{cases}
        \displaystyle \bar \vx \in \argmin_{\vx \in \sX}~l^{(T)}(\vx, \bar \vy), \\
        \displaystyle \bar \vy \in \sY^*(\bar \vx).
    \end{cases}
    \label{eq:T-step-cournot}
\end{equation}
It is a Nash-equilibrium problem: no follower has an incentive to change their decisions because $\bar \vy \in \sY^*(\bar \vx)$; the leader also cannot further reduce $l^{(T)}(\vx, \bar \vy)$ given $\bar \vy$.

\textbf{$T$-step monopoly model.}
The leader dictates a decision $\vy \in \sY$ for the followers but allows them to update their decisions $T$ steps based on $\vy \in \sY$. To find the optimal decision $\hat \vx$ and $\hat \vy$, we need to solve
\begin{equation}
     \hat \vx, \hat \vy \in \argmin_{\vx \in \sX,\, \vy \in \sY.}~l^{(T)}(\vx, \vy).
\label{eq:T-step-monopoly}
\end{equation}
It is a single-level optimization problem with the leader's and the followers' decisions optimized altogether.

\subsection{Optimality gaps}
\label{sec:gap}

We first give the following proposition, which holds true regardless of the form of $h(\vx, \vy)$ or the properties of  $f(\vx, \cdot)$. 

\begin{proposition}
\label{prop:model-gap}
Suppose that the function $h: \sX \to \sY$ used for formulating Problems \eqref{eq:T-step-cournot} and \eqref{eq:T-step-monopoly} satisfies the following condition
\begin{equation}
    \lim_{t \to \infty} h^{(t)}(\vx, \vy) \in \sY^*(\vx), \quad \forall \vx \in \sX, \ \vy \in \sY.
    \label{eq:condition-h}
\end{equation}
Then given any $(\bar \vx, \bar \vy)$ and $(\hat \vx, \hat \vy)$ that solve Problems \eqref{eq:T-step-cournot} and \eqref{eq:T-step-monopoly}, respectively, and $(\vx^*, \vy^*)$ that solves Problem \eqref{eq:bilevel}, we have
$
    l^{(T)}(\hat \vx, \hat \vy) \leq l(\vx^*, \vy^*) \leq  l^{(T)}(\bar \vx, \bar \vy).
$
\end{proposition}
\begin{proof}
See Appendix \ref{app:gap} for the proof. 
\end{proof}

The following assumptions are needed in our analysis. 

\begin{assumption}
\label{ass:upper}
The set $\sX$ is closed and convex; $l(\vx, \vy)$ is twice continuously differentiable; we have $\|\nabla_{\vx} l(\vx, \vy)\|_2 \leq G_x$, $\| \nabla_{\vy} l(\vx, \vy)\|_2 \leq G_y$,  $\|\nabla_{\vx \vy} l(\vx, \vy) \|_2 \leq G_{xy}$, and $\|\nabla_{\vy \vy} l(\vx, \vy) \|_2 \leq G_{yy}$ for all $\vx \in \sX$ and $\vy \in \sY$.
\end{assumption}

\begin{assumption}
\label{ass:lower}
The set $\sY$ is closed and convex; $f(\vx, \vy)$ is continuously differentiable; there exists $\gamma > 0$ such that $f(\vx, \cdot)$ is $\gamma$-strongly monotone for all $\vx \in \sX$; we have $\|\nabla_{\vy} f(\vx, \vy)\|_2 \leq L_y$ and $\|\nabla_{\vx} f(\vx, \vy)\|_2 \leq L_x$ for all $\vx \in \sX$ and $\vy \in \sY$.
\end{assumption}

\begin{assumption}
\label{ass:dynamics}
The function $h(\vx, \vy)$ is twice continuously differentiable; the intrinsic parameter $r < 2\gamma / L_y^2$; we have $\|\nabla_{\vx} h(\vx, \vy)\|_2 \leq H_x$, $\|\nabla_{\vx \vy} h(\vx, \vy) \|_2 \leq H_{xy}$, and $\|\nabla_{\vy \vy} h(\vx, \vy) \|_2 \leq H_{yy}$ for all $\vx \in \sX$ and $\vy \in \sY$.
\end{assumption}

Assumptions \ref{ass:upper} and \ref{ass:lower} are similar to those adopted by \citet{ghadimi2018approximation} and \citet{hong2020two} with one notable difference:  they require strong convexity in the lower-level optimization problem whereas we require strong monotonicity in the VI. According to \citet{hiriart1982points}, $h(\vx, \vy)$ is differentiable as long as the boundary of $\sY$ is smooth; without this condition, $h(\vx, \vy)$ is still \emph{almost everywhere} differentiable 
\citep{rademacher1919partielle}. Under Assumptions \ref{ass:lower} and \ref{ass:dynamics}, Proposition \ref{thm:convergence-strong} ensures Condition \eqref{eq:condition-h}  be satisfied. Also, according to Proposition \ref{prop:implicit-differentiation}, a differentiable implicit function  $y^*(\vx)$ that maps $\vx$ to $\sY^*(\vx)$ can be defined. The objective function of Problem \eqref{eq:bilevel} then can be rewritten as $l^*(\vx) = l(\vx, y^*(\vx))$. Under these assumptions, Example \ref{eg:nash} gives the VI formulation of the two models.

\begin{proposition}
\label{prop:vi-cournot}
Suppose that $(\bar \vx, \bar \vy)$ solves the $T$-step Cournot game \eqref{eq:T-step-cournot}, then for all $(\vx, \vy) \in \sX \times \sY$, we have
\begin{equation}
    \langle \nabla_{\vx} l^{(T)}(\bar \vx, \bar \vy), \vx - \bar \vx \rangle + \langle f(\bar \vx, \bar \vy), \vy - \bar \vy \rangle \geq 0.
    \label{eq:cournot-vi}
\end{equation}
The converse also holds if $l^{(T)}(\vx, \vy)$ is convex in $\vx$.
\end{proposition}

\begin{proposition}
\label{prop:vi-monopoly}
Suppose that $(\hat \vx, \hat \vy)$ solves the $T$-step monopoly model \eqref{eq:T-step-monopoly}, then for all  $(\vx, \vy) \in \sX \times \sY$, we have
\begin{equation}
    \langle \nabla_{\vx} l^{(T)}(\hat \vx, \hat \vy), \vx - \hat \vx \rangle + \langle \nabla_{\vy} l(\hat \vx, \hat \vy), \vy - \hat \vy \rangle \geq 0.
    \label{eq:monopoly-vi}
\end{equation}
The converse also holds if $l^{(T)}(\vx, \vy)$ is convex in $(\vx, \vy)$.
\end{proposition}

Next, we will confirm $T$-step Cournot games and monopoly models respectively provide an upper and a lower bound to Problem \eqref{eq:bilevel}. More importantly, we will prove that the gap between the bounds provided by the two models converges to 0 at a fast rate as $T$ increases, which implies that a small $T$ would make a good approximation. 
\begin{theorem}
\label{thm:main}
Under Assumptions \ref{ass:upper}-\ref{ass:dynamics}, suppose that $l^*(\vx)$ is $\mu$-strongly convex on $\sX$ and denote $\vx^*$ as the unique minimizer of $l^*(\vx)$ on $\sX$. We write $\eta = 1 - 2 \gamma r  + r^2 L_y^2$, $G_l = G_x + L_x G_y$ and, without loss of generality, assume $\mu = 1$ and $\gamma = 1$. For any $(\bar \vx, \bar \vy)$ that solves Problem \eqref{eq:T-step-cournot}, denoting $\bar \delta^T = l^{(T)}(\bar \vx, \bar \vy) - l^*(\vx^*)$, we then have 
\begin{equation}
    0 \leq \bar \delta^T \leq G_l \cdot G_y L_x \cdot  {\color{red} \eta^{T/2}}.
    \label{eq:cournot-rate}
\end{equation}
Further assuming $d_{\max} = \diam{(\sY)} < \infty$, given any $(\hat \vx, \hat \vy)$ that solves the Problem \eqref{eq:T-step-monopoly}, denoting $\hat \delta^T = l^*(\vx^*) - l^{(T)}(\hat \vx, \hat \vy)$, we then have
\begin{equation}
\begin{split}
   0 \leq \hat \delta^T \leq M \cdot  G_l \cdot \eta^{T / 2}  + G_y d_{\max} \cdot (1 + H_{xy}  + L_x H_{yy} ) \cdot  {\color{red} (T + 1)  \cdot \eta^{T / 2}},
   \label{eq:monopoly-rate}
\end{split}
\end{equation}
where $M = 2 G_y L_x  + G_{xy} d_{\max}  + 2 L_x G_{yy} d_{\max} +  G_y H_x$.
\end{theorem}
\begin{proof}
Proposition \ref{prop:model-gap} directly guarantees $\bar \delta^T \geq 0$ and $\hat \delta^T \geq 0$; the remaining proof is given in Appendix \ref{app:gap}.    
\end{proof}

Based on Theorem \ref{thm:main}, we may develop the following simple procedure to obtain an approximated feasible solution to Problem \eqref{eq:bilevel}. (1) Select an appropriate $T \in \sN$. (2) Solve the two VI problems \eqref{eq:cournot-vi} and \eqref{eq:monopoly-vi} to obtain $(\bar \vx, \bar \vy)$ and $(\bar \vx, \bar \vy)$, respectively, and then calculate $\delta^T = l^{(T)}(\bar \vx, \bar \vy) - l^{(T)}(\hat \vx, \hat \vy)$. (3) If $\delta^T$ is small enough, accept $(\bar \vx, \bar \vy)$ as an approximated feasible solution; otherwise, increase $T$ and return to Step (2). We expect the above procedure to converge quickly, as Theorem \ref{thm:main} guarantees the gap decreases at an exponential rate as $T$ increases.  %

\subsection{Solution algorithms}
\label{sec:algorithm}

The two VI problems, \eqref{eq:cournot-vi} and \eqref{eq:monopoly-vi}, can be solved using the projection method, see Algorithms \ref{alg:cournot} and \ref{alg:monopoly} respectively, for 
the pseudocodes.
 
\begin{algorithm}[H]
   \caption{\small Solving $T$-step Cournot game. Input: $\vx^0 \in \sX$, $\vy^0 \in \sY$, step size $\alpha$. Output: $\bar \vx \in \sX$ and $\bar \vy \in \sY$.}
   \label{alg:cournot}
\begin{algorithmic}[1]
{\small
   \FOR{$t = 0, 1, \ldots$}
   \STATE Calculate $l_{\vx} = \nabla_{\vx} l^{(T)}(\vx^t, \vy^t)$.
    \STATE Set $\displaystyle \vx^{t + 1} = \argmin_{\vx \in \sX}~\alpha \cdot \langle l_{\vx}, \vx - \vx^{t} \rangle + \|\vx - \vx^t\|_2$ and $\vy^{t + 1} = h^{(T)}(\vx^{t}, \vy^t)$.
    \STATE After convergence, break and return $(\bar \vx, \bar \vy) = (\vx^t, \vy^t)$.
   \ENDFOR
}
\end{algorithmic}
\end{algorithm}
\vspace{-15pt}
\begin{algorithm}[H]
   \caption{\small Solving $T$-step monopoly model. Input: $\vx^0 \in \sX$, $\vy^0 \in \sY$, step sizes $\alpha$ and $\beta$. Output: $\hat \vx \in \sX$ and $\hat \vy \in \sY$.}
   \label{alg:monopoly}
\begin{algorithmic}[1]
 {\small
   \FOR{$t = 0, 1, \ldots$}
   \STATE Calculate $l_{\vx} = \nabla_{\vx} l^{(T)}(\vx^t, \vy^t)$ and $l_{\vy} = \nabla_{\vy} l^{(T)}(\vx^t, \vy^t)$.
    \STATE Set $\displaystyle \vx^{t + 1} = \argmin_{\vx \in \sX}~\alpha \cdot \langle l_{\vx}, \vx - \vx^{t} \rangle + \|\vx - \vx^t\|_2$ and
    $\displaystyle
            \vy^{t + 1} = \argmin_{\vy \in \sY}~\beta \cdot \langle l_{\vy}, \vy - \vy^{t} \rangle + \|\vy - \vy^t\|_2.
    $
    \label{ln:monopoly-lower}
     \STATE After convergence, break and return $(\hat \vx, \hat \vy) = (\vx^t, \vy^t)$.
   \ENDFOR
    }
\end{algorithmic}
\end{algorithm}

In the two algorithms, the leader and the followers evolve together, meaning they each update strategies simultaneously in every step of a shared evolution process. 
Hence, they break the bilevel hierarchy and turn the solution process into a single loop. At each iteration of that single loop, we need to calculate the gradients of $l^{(T)}(\vx, \vy) = l(\vx, h^{(T)}(\vx, \vy))$, which has an explicit expression. Hence, we may directly calculate its gradient via AD. To this end, we first need to feed the function $h(\vx, \vy)$ as a differentiable layer into differentiable-programming frameworks, e.g., TensorFlow \citep{abadi2016tensorflow} or PyTorch \citep{paszke2019pytorch}. As discussed earlier, using the package \texttt{cvxpylayers} to code $h(\vx, \vy)$ is always an option. In many cases, however, $h(\vx, \vy)$ can be directly coded as a differentiable program built by elementary operations only (no differentiable optimization solver is involved). In such cases, AD guarantees the complexity of calculating the gradient of $l^{(T)}(\vx, \vy)$ can be tightly bounded by that of calculating $l^{(T)}(\vx, \vy)$ itself. We refer the readers to Appendix \ref{app:dp-gradient} for more details.

\textbf{Summary.} To summarize, our main result is that Problem \eqref{eq:bilevel} can be tightly and efficiently bounded by two easier problems,  which can be solved by AD-based first-order methods. The proposed approximation scheme thus promises to significantly reduce the computational overhead for obtaining high-quality local solutions to Problem \eqref{eq:bilevel}. 

\textbf{Comparison.} The proposed framework bridges two general classes of schemes for approximating bilevel programs proposed under different contexts.  The first class includes several classic heuristics for solving Stackelberg congestion games (SCGs) (see Section \ref{sec:stackelberg}) dated back to the 1970s. For example, \citet{dantzig1979formulating} proposed to solve an SCG by assuming the followers work cooperatively with the leader to achieve the system-optimal state, which essentially turns the leader into a ``dictator". Another is \citet{tan1979hybrid}'s algorithm that finds the best response of the leader and followers iteratively while holding each other's decisions as a fixed input so that the leader and the followers are competing ``\`a la Cournot". Our scheme extends the behavior assumptions underlying these two.

The second class is AD-based approximation schemes. The proposed algorithms are similar to one-stage AD (see Section \ref{sec:ad}) that are originally motivated by ML applications, e.g., neural architecture search \citep{liu2018darts}. One-stage AD and our algorithm are structurally similar;  the difference lies in how the lower-level solution (the followers' decision) is updated. In our scheme, the follower may move $T$ steps according to either their own interest ($T$-step Cournot) or the leader's mandate ($T$-step monopoly). The former may be viewed as a natural extension of one-stage AD; the latter, however, represents a novel behavior interpretation. With this interpretation, our scheme yields adjustable upper and lower bounds on the original problem.

By bringing together the  two classes of approximation schemes motivated by different applications from different disciplines, our work provides new insights for both. 
On the one hand, it shows that classical game-theoretic approximations for bilevel programs can be substantially refined with modest computational efforts, using ideas borrowed from the ML community. On the one hand, it provides a theoretical justification to AD-based methods from the lens of game theory.

\textbf{Extensions.} We leave the discussion of several extensions to the appendix. In Appendix \ref{app:other-h}, we discuss other choices of the form of $h(\vx, \vy)$ to formulate two proposed models. In Appendix \ref{app:global}, we discuss how to search for the global solution efficiently when the global convexity assumption on $l^*(\vx)$ is relaxed. In Appendix \ref{sec:non-unique}, we address the case when the lower-level VI admits multiple solutions by developing a heuristic based on Algorithms \ref{alg:cournot} and \ref{alg:monopoly}. 

\section{Numerical Examples}
\label{sec:experiment}

To validate our analytical insights, we test our framework on the Stackelberg duopoly (cf. Example \ref{eg:duopoly}). We solve the $T$-step Cournot duopoly \eqref{eq:cournot-duopoly} and the $T$-step monopoly \eqref{eq:monopoly-duopoly} via Algorithms \ref{alg:cournot} and \ref{alg:monopoly}, respectively, with $r = 0.4$ and $T = 0, \ldots, 4$. Firm A's profits are then reported in Table \ref{tab:duopoly}, which indicates that, as $T$ increases, Firm A's optimal profit generated by either new model quickly converges to that by the Stackelberg duopoly. The decreasing rate also meets the expectation set by Theorem \ref{thm:main}.

\begin{table}[ht]
  \centering
  \caption{Firm A's profit in $T$-step Cournot duopoly and $T$-step monopoly models with different $T$s.}
  \label{tab:duopoly}%
    \centering
  \footnotesize
  \vspace{2pt}
    \begin{tabular}{cccccc}
    \toprule
    $T$     & 0     & 1     & 2     & 3     & 4\\
    \midrule
    $T$-step Cournot & 0.111 & 0.124 & 0.125 & 0.125 & 0.125 \\
    $T$-step monopoly & 0.250 & 0.150 & 0.130 & 0.126 & 0.125\\
    \bottomrule
    \end{tabular}%
\end{table}%

The results of other experiments are reported in Appendix \ref{app:experiments}. These additional experiments are designed to test the proposed scheme on larger and harder problems that arise from diverse applications and to compare it against benchmark bilevel algorithms.

\section{Conclusion}
\label{sec:conclusion}

It is well known that non-cooperative games may lead to inefficient outcomes. Caused by the lack of cooperation between self-interested agents, this loss of efficiency --- also known as the price of anarchy --- is best illustrated by the Braess paradox \citep{braess1968paradoxon} in transportation. 
Nevertheless, \citet{roughgarden2002bad} proves that the total travel time experienced by travelers at \textit{user equilibrium} (UE) is tightly bounded from the above by that  achieved when they are fully cooperative, a state called system optimum (SO). Evidently, the UE state corresponds to the outcome of a \textit{Cournot game} because everyone competes equally, whereas the SO state can be brought about only if the choices of all travelers are \textit{monopolized}. 
The finding in this paper indicates that the gap between ``Cournot" and ``monopolized" states can not only  be bounded but also be \textit{narrowed} by simultaneously giving the leader a limited ability of anticipation in the Cournot game and the followers  limited ``freedom"  to pursue their own interests in the monopoly model. Moreover, given the right conditions, they both can converge to the outcome of a Stackelberg game, where a compromise is arranged: the leader cannot dictate the followers' choice but can influence it indirectly; the followers enjoy the freedom of choice but must heed the leader's guidance. 

The models proposed as approximations of the Stackelberg game (bilevel program) are solved by AD-based methods. This connects our work to many bilevel programming algorithms developed in the ML literature \citep{liu2021investigating}. Our theoretical results help answer a question that has been extensively debated recently: when and why can AD-based approximation schemes deliver satisfactory solutions? 
Our work contributes to this debate by revealing that --- besides the power of AD itself \citep{ablin2020super} --- the underlying game-theoretic structure plays an important role in shaping the power of these methods.  We hope this finding will inspire more interdisciplinary works across the domains that count bilevel programming in their toolbox.

\appendix
\numberwithin{algorithm}{section}
\numberwithin{equation}{section}
\numberwithin{figure}{section}
\numberwithin{table}{section}

\section{Proofs in Section \ref{sec:main-results}}

\subsection{Proof of Proposition \ref{prop:model-gap}}

\begin{proof}[Proof of Proposition \ref{prop:model-gap}]
First, $(\bar \vx, \bar \vy)$ is a feasible solution to Problem \eqref{eq:bilevel} since $(\bar \vx, \bar \vy) \in \{(\vx, \vy): \vx \in \sX, \vy \in \sY^*(\vx) \}$. Thus, we can directly conclude
$
    l(\vx^*, \vy^*) \leq l(\bar \vx, \bar \vy)  = l^{(T)}(\bar \vx, \bar \vy).
$
Next, as $l(\vx, \vy) = l(\vx, h^{(T)}(\vx, \vy)) = l^{(T)}(\vx, \vy)$ for all $ (\vx, \vy) \in \{(\vx, \vy): \vx \in \sX, \vy \in \sY^*(\vx) \}$, we claim $(\vx^*, \vy^*)$ also solves the following optimization
\begin{equation}
     \vx^*, \vy^* \in \argmin_{\vx \in \sX,\, \vy \in \sY^*(\vx).}~l^{(T)}(\vx, \vy).
\end{equation}
Eventually, $\{(\vx, \vy): \vx \in \sX, \vy \in \sY^*(\vx) \} \subseteq \sX \times \sY$ implies that $l^{(T)}(\hat \vx, \hat \vy) \leq l^{(T)}(\vx^*, \vy^*) = l(\vx^*, \vy^*)$.
\end{proof}

\subsection{Proof of Theorem \ref{thm:main}}
\label{app:gap}

We first provide the following lemmas.

\begin{lemma}
\label{prop:tired}
If $l^*(\vx) = l(\vx, y^*(\vx))$ is $\mu$-strongly convex on $\vx \in \sX$, then we have
\begin{equation}
    \mu \cdot \|\vx^* - \vx\|_2^2 \leq  \langle \nabla l^*(\vx), \vx - \vx^* \rangle, \quad \forall \vx \in  \sX.
    \label{eq:add2}
\end{equation}
\end{lemma}

\begin{proof}
As $l^*(\vx)$ is $\mu$-strongly convex, we have
\begin{equation}
    \begin{cases}
    \displaystyle l^*(\vx^*) \geq l^*( \vx) + \langle \nabla l^*( \vx), \vx^* -  \vx \rangle + \frac{\mu}{2} \cdot \|\vx^* -  \vx\|_2^2, \\[5pt]
    \displaystyle l^*( \vx) \geq l^*(\vx^*) + \langle \nabla l^*(\vx^*),  \vx - \vx^* \rangle + \frac{\mu}{2} \cdot \|\vx^* -  \vx\|_2^2.
    \end{cases}
\end{equation}
Adding these two equations, we then have
\begin{equation}
    0 \geq \langle \nabla l^*( \vx), \vx^* -  \vx \rangle + \langle \nabla l^*( \vx^*),  \vx - \vx^* \rangle + \mu \cdot \|\vx^* -  \vx\|_2^2 \geq  \langle \nabla l^*( \vx), \vx^* -  \vx \rangle + \mu \cdot \|\vx^* -  \vx\|_2^2,
\end{equation}
where the second inequality comes from the fact that $\vx^*$ minimizes $l^*(\vx)$.
\end{proof}

\begin{lemma}
\label{prop:y-star-lip}
For all $\vx \in \sX$, we have
$
    \|\nabla y^*(\vx)\|_2 \leq L_x / \gamma.
$
\end{lemma}
\begin{proof}
Let $\vx$ and $\vx'$ be two arbitrary points in $\sX$. Denote $\vy^* = y^*(\vx)$ and $\vy^{\prime *} = y^*(\vx')$. Then we have
\begin{equation}
    \begin{cases}
    \langle f(\vx, \vy^*), \vy^{\prime *} - \vy^* \rangle \geq 0, \\[2pt]
    \langle f(\vx', \vy^{\prime *}), \vy^* - \vy^{\prime *} \rangle \geq 0.
    \end{cases}
\end{equation}
Adding these two inequalities, we then have
\begin{equation}
    \langle f(\vx', \vy^{\prime *}) - f(\vx, \vy^{\prime *}) + f(\vx, \vy^{\prime *}) - f(\vx, \vy^*), \vy^* - \vy^{\prime *} \rangle \geq 0.
    \label{eq:add-two-again}
\end{equation}
We can further obtain that
\begin{equation}
\begin{split}
    \langle f(\vx', \vy^{\prime *}) - f(\vx, \vy^{\prime *}), \vy^* - \vy^{\prime *} \rangle \geq \langle f(\vx, \vy^*) - f(\vx, \vy^{\prime *}), \vy^* - \vy^{\prime *} \rangle \geq \gamma \cdot  \|\vy^* - \vy^{\prime *}\|_2^2,
\end{split}
\end{equation}
where the first inequality comes from \eqref{eq:add-two-again} and the second inequality comes from the assumption that $f$ is strongly monotone with respect to $\|\cdot\|_2$.
Noting that $f(\cdot, \vy)$ is $L_x$-Lipschitz continuous and using the Cauchy-Schwartz inequality, we eventually have
\begin{equation}
    \|\vy^* - \vy^{\prime *}\|_2 \leq \frac{1}{\gamma} \cdot  \| f(\vx', \vy^{\prime *}) - f(\vx, \vy^{\prime *})\|_2 \leq \frac{L_x}{\gamma} \cdot  \|\vx - \vx'\|_2.
\end{equation}
Letting $\vx' \to \vx$, we then conclude the proof.  
\end{proof}

\begin{lemma}
\label{prop:l-star-lip}
The function $l^*(\vx)$ is $(G_x + L_x G_y / \gamma)$-Lipschitz continuous with respect to $\|\cdot\|_2$.
\end{lemma}
\begin{proof}
Noting that $l(\vx, \cdot)$, $l(\cdot, \vy)$ and $y^*(\vx)$ are respectively $G_y$-, $G_x$- and $L_x / \gamma$-Lipschitz continuous, by directly applying the triangle inequality of norms, we can obtain that
\begin{equation}
\begin{split}
    \|\nabla l^*(\vx)\| &= \|\nabla_{\vx} l(\vx, y^*(\vx)) + \nabla y^*(\vx) \cdot \nabla_{\vy} l(\vx, y^*(\vx)) \|\\
    &\leq \| \nabla_{\vx} l(\vx, y^*(\vx))\| + \|\nabla y^*(\vx)\| \cdot \|\nabla_{\vy} l(\vx, y^*(\vx))\| \leq G_x + \frac{L_x G_y}{\gamma}.
\end{split}
\end{equation}
We thus conclude the proof.
\end{proof}

\begin{lemma}
We can directly obtain the following formulas by applying the chain rule
\begin{align}
    \nabla l^*(\vx) &= \nabla_{\vx} l(\vx, y^*(\vx)) + \nabla y^*(\vx) \cdot \nabla_{\vy} l(\vx, y^*(\vx)), \label{eq:nabla-l-star} \\
    \nabla_{\vx} l^{(T)}(\vx, \vy) &= \nabla_{\vx} l(\vx, h^{(T)}(\vx, \vy)) + \nabla_{\vx} h^{(T)}(\vx, \vy) \cdot \nabla_{\vy} l(\vx, h^{(T)}(\vx, \vy)), \label{eq:nabla-l-T}\\
    \nabla_{\vx} h^{(t)}(\vx, \vy) &= \nabla_{\vx} h( \vx, h^{(t - 1)}(\vx, \vy)) + \nabla_{\vx} h^{(t - 1)}( \vx,  \vy) \cdot \nabla_{\vy} h( \vx,  h^{(t - 1)}(\vx, \vy)), \label{eq:nabla-h-T} \\
    \nabla_{\vy} h^{(t)}(\vx, \vy) &= \nabla_{\vy} h(\vx, h^{(t - 1)}(\vx, \vy)) \cdot \nabla_{\vy} h^{(t - 1)}(\vx, \vy). \label{eq:nabla-h-T-y}
\end{align}
\end{lemma}

\begin{proposition}
\label{prop:key-result}
For any $\vx \in \sX$ and $\vy^* = y^*(\vx)$, we have
\begin{equation}
    \| \nabla y^*(\vx) - \nabla_{\vx} h^{(T)}(\vx, \vy^*)\|_2 \leq \frac{L_x}{\gamma} \cdot \eta^{T/2}.
\end{equation}
\end{proposition}

\begin{proof}
By reformulating Equation \eqref{eq:nabla-y-star} in Proposition \ref{thm:convergence-strong}, we have
\begin{equation}
        \nabla y^*(\vx) = \nabla_{\vx} h(\vx,  \vy^*) + \nabla y^*(\vx) \cdot \nabla_{\vy} h(\vx,  \vy^*).
        \label{eq:y-star-iter}
\end{equation}
Based on Equations \eqref{eq:nabla-h-T}, \eqref{eq:nabla-h-T-y}, and \eqref{eq:y-star-iter}, we can then iteratively obtain 
\begin{equation}
\begin{split}
    \nabla y^*( \vx) - \nabla_{\vx} h^{(T)}( \vx,   \vy^*) &= \left( \nabla y^*( \vx) - \nabla_{\vx} h^{(T - 1)}( \vx,   \vy^*) \right)\cdot \nabla_{\vy} h^{(1)}( \vx,  \vy^*) \\
    &= \left( \nabla y^*( \vx) - \nabla_{\vx} h^{(T - 2)}( \vx,  \vy^*) \right)\cdot \nabla_{\vy} h^{(2)}( \vx,  \vy^*) \\
    &= \cdots = \nabla y^*( \vx) \cdot \nabla_{\vy} h^{(T)}( \vx,  \vy^*).
\end{split}
\label{eq:middle-2}
\end{equation}
Using the triangle inequality of norms and Applying Lemma \ref{prop:y-star-lip}, we then obtain
\begin{equation}
    \| \nabla y^*(\vx) - \nabla_{\vx} h^{(T)}(\vx, \vy^*)\|_2 \leq \|\nabla y^*( \vx) \|_2 \cdot \|\nabla_{\vy} h^{(T)}( \vx,  \vy^*)\|_2 \leq \frac{L_x}{\gamma} \cdot \|\nabla_{\vy} h^{(T)}( \vx,  \vy^*)\|_2.
\end{equation}
We eventually conclude the proof by applying Proposition \ref{thm:convergence-strong}.
\end{proof}

The remaining proof will be decomposed into two parts,  Part I and Part II, in which we will prove the upper bounds given to $\bar \delta^T = l^{(T)}(\bar \vx, \bar \vy) - l^*(\vx)$ and $\hat \delta^T = l^*(\vx) - l^{(T)}(\hat \vx, \hat \vy)$, respectively.

\subsubsection{Part I of the Remaining Proof}
\label{app:gap-cournot}

By applying Lemma \ref{prop:l-star-lip}, we can obtain that
\begin{equation}
    l(\bar \vx, \bar \vy) - l^*(\vx^*) = l^*(\bar \vx) - l^*(\vx^*) \leq \left(G_x + \frac{L_x G_y}{\gamma}\right) \cdot \|\bar \vx - \vx^*\|_2
    \label{eq:cournot-first-step}
\end{equation}
Therefore, we only need to bound $\|\bar \vx- \vx^*\|_2$. Based on Proposition \ref{prop:vi-cournot}, we have
\begin{equation}
    \langle \nabla_{\vx} l^{(T)}(\bar \vx, \bar \vy), \vx^* - \bar \vx \rangle \geq 0.
\label{eq:exact-condition}
\end{equation}
Then we can set $\vx = \bar \vx$ in Equation \eqref{eq:add2} and combine it with Equation \eqref{eq:exact-condition}, which gives
\begin{equation}
    \mu \cdot \|\vx^* - \bar \vx\|_2^2 \leq \langle \nabla l^*(\bar \vx), \bar \vx  - \vx^*\rangle + \langle \nabla_{\vx} l^{(T)}(\bar \vx, \bar \vy), \vx^* - \bar \vx\rangle = \langle \nabla l^*(\bar \vx)  - \nabla_{\vx} l^{(T)}(\bar \vx, \bar \vy), \bar \vx  - \vx^*\rangle
\end{equation}
Using the Cauchy-Schwartz inequality, we then obtain
\begin{equation}
    \|\vx^* - \bar \vx\|_2 \leq \frac{1}{\mu}  \cdot \| \nabla l^*(\bar \vx)  - \nabla_{\vx} l^{(T)}(\bar \vx, \bar \vy)\|_2.
\end{equation}
Noting that $\bar \vy = y^*(\bar \vx) = h^{(T)}(\bar \vx, \bar \vy)$, substituting $\nabla l^*(\bar \vx)$ and $\nabla_{\vx} l^{(T)}(\bar \vx, \bar \vy)$ by Equations \eqref{eq:nabla-l-star} and \eqref{eq:nabla-l-T}, and then applying Proposition \ref{prop:key-result}, we eventually have
\begin{equation}
\begin{split}
    \|\vx^* - \bar \vx\|_2 &\leq \frac{1}{\mu} \cdot
    \| \nabla  y^*(\bar \vx) - \nabla_{\vx} h^{(T)}(\bar \vx, \bar \vy)\|_2 \cdot \|\nabla_{\vy} l(\bar \vx, \bar \vy)\|_2 \leq \frac{G_y L_x}{\mu \gamma} \cdot \eta^{T/2}.
    \label{eq:middle-1}
\end{split}
\end{equation}
We can hence conclude the proof by combining Equations \eqref{eq:cournot-first-step} and \eqref{eq:middle-1}.

\subsubsection{Part II of the Remaining Proof}
\label{app:gap-monopoly}

We need an additional lemma to prove the upper bound given to $\hat \delta^T = l^*(\vx) - l^{(T)}(\hat \vx, \hat \vy)$.

\begin{lemma}
\label{prop:key-again}
For any $\vx \in \sX$ and $\vy \in \sY$, we have
\begin{equation}
    \|\nabla_{\vx} h^{(T)}(\vx, \vy) - \nabla y^*(\vx) \|_2 \leq \left( \|\nabla_{\vx} h(\vx, \vy)\|_2 + \delta^0 T \left( H_{xy}  + \frac{L_x}{\gamma} \cdot H_{yy} \right) \right)\eta^{T / 2}, 
\end{equation}
where $\delta^0 = \|\vy - y^*(\vx)\|_2$.
\end{lemma}
\begin{proof}
In the sequel, we write $\vy^t = h^{(t)}(\vx, \vy)$, $\mD^{(t)} = \nabla_{\vx} h^{(t)}(\vx, \vy^0) - \nabla y^*(\vx)$, $\mE^{(t)} = \nabla_{\vx} h(\vx, \vy^{t - 1}) - \nabla_{\vx} h(\vx,  \vy^*)$ and $\mF^{(t)} = \nabla_{\vy} h(\vx, \vy^{t - 1}) - \nabla_{\vy} h(\vx,  \vy^*)$,
$\mH^{(t)} = \nabla_{\vy} h^{(t)}(\vx, \vy^{T - t})$ and $\mJ = \nabla y^*(\vx)$. By recursively applying \eqref{eq:nabla-h-T} and \eqref{eq:y-star-iter}, we then have
\begin{equation}
\begin{split}
    \mD^{(T)} &= \mD^{(T - 1)} \cdot \mH^{(1)} + \mE^{(T)} +  \mJ \cdot \mF^{(T)}\\
    &= \mD^{(T - 2)} \cdot \mH^{(2)} + \mE^{(T - 1)} \cdot \mH^{(1)} + \mE^{(T)} + \mJ \cdot (\mF^{(T - 1)} \cdot \mH^{(1)} + \mF^{(T)}) \\
    &= \cdots = \mD^{(0)} \cdot \mH^{(T)} + \sum_{t = 0}^T \mE^{(T - t)} \cdot \mH^{(t)} + \mJ \cdot \sum_{t = 0}^T \mF^{(T - t)} \cdot \mH^{(t)}
\end{split}
\end{equation}
Applying the triangle inequality of norms, we then have
\begin{equation}
    \|\mD^{(T)}\|_2 
    \leq \|\mD^{(0)}\|_2 \cdot \|\mH^{(T)}\|_2 + \sum_{t = 0}^T (\|\mE^{(T - t)} \|_2  + \|\mJ\|_2 \cdot \|\mF^{(T - t)}\|_2) \cdot \|\mH^{(t)}\|_2.
\end{equation}
Lemma \ref{prop:y-star-lip} and Proposition \ref{thm:convergence-strong} then imply that  $\|\mJ\|_2 \leq L_x / \gamma$, $\|\mE^{(t)}\|_2 \leq H_{xy} \cdot \|\vy^{t - 1} - \vy^*\| \leq H_{xy} \delta^0 \eta^{t / 2}$, $\|\mF^{(t)}\|_2 \leq H_{yy} \cdot \|\vy^{t - 1} - \vy^*\| \leq H_{yy}  \delta^0 \eta^{t / 2} $, and also $\|\mH\|^{t} \leq \eta^{t / 2}$. Thus, we eventually have
\begin{equation}
\begin{split}
    \|\mD^{(T)}\|_2 
    &\leq \eta^{t / 2} \|\mD^{(0)}\|_2 + \sum_{t = 0}^T (H_{xy} \delta^0 \eta^{(T - t) / 2}  + \frac{L_x}{\gamma} \cdot H_{yy} \delta^0 \eta^{(T - t) / 2} ) \cdot \eta^{t / 2} \\
    &= \left( \|\mD^{(0)}\|_2 + \delta^0 (T + 1) \left( H_{xy}  + \frac{L_x}{\gamma} \cdot H_{yy} \right) \right)\eta^{T / 2}.
\end{split}
\end{equation}

\end{proof}

Now we are ready to complete the proof.
Using the triangle inequality and applying Lemma \ref{prop:l-star-lip}, we obtain
\begin{equation}
\begin{split}
    l^*(\vx^*) - l(\hat \vx, h^{(T)}(\hat \vx, \hat \vy)) &\leq |l^*(\vx^*) - l^*(\hat \vx)| + |l(\hat \vx, y^*(\hat \vx)) - l(\hat \vx, h^{(T)}(\hat \vx, \hat \vy))|_2 \\
    &\leq  \left(G_x + \frac{L_x G_y}{\gamma}\right) \cdot \|\hat \vx - \vx^*\|_2 + G_y \cdot \|y^*(\hat \vx) - h^{(T)}(\hat \vx, \hat \vy)\|_2
    \label{eq:M-1}
\end{split}
\end{equation}
We first bound $\| y^*(\hat \vx) - h^{(T)}(\hat \vx, \hat \vy)\|_2$. We can obtain that
\begin{equation}
    \begin{split}
        \|y^*(\hat \vx) - h^{(T)}(\hat \vx, \hat \vy)\|_2 \leq \eta^{T/2} \cdot  \| \hat \vy - y^*(\hat \vx)\|_2 \leq d_{\max} \cdot \eta^{T/2},
    \end{split}
    \label{eq:h-convergence-up}
\end{equation}
where $d_{\max} = \diam{(\sY)} < \infty$.

We then bound $\|\hat \vx - \vx^*\|$. Based on Proposition \ref{prop:vi-monopoly}, we have
\begin{equation}
    \langle \nabla_{\vx} l^{(T)}(\hat \vx, \hat \vy), \vx^* - \hat \vx \rangle \geq 0.
    \label{eq:tired}
\end{equation}
Combining it with Equation \eqref{eq:add2} by setting $\vx = \hat \vx$ and using the Cauchy-Schwartz inequality, we then have
\begin{equation}
\begin{split}
        &\mu \cdot \|\vx^* - \hat \vx\|_2^2 \\ &\leq \langle \nabla l^*(\hat \vx) - \nabla_{\vx} l^{(T)}(\hat \vx, \hat \vy), \hat \vx  - \vx^*\rangle \\
        &= \langle \nabla l^*(\hat \vx) - \nabla_{\vx} l^{(T)}(\hat \vx, y^*(\hat \vx)) + \nabla_{\vx} l^{(T)}(\hat \vx, y^*(\hat \vx)) -  \nabla_{\vx} l^{(T)}(\hat \vx, \hat \vy), \hat \vx  - \vx^*\rangle \\
        &\leq (\| \nabla l^*(\hat \vx) - \nabla_{\vx} l^{(T)}(\hat \vx, y^*(\hat \vx))\|_2 + \| \nabla_{\vx} l^{(T)}(\hat \vx, y^*(\hat \vx)) - \nabla_{\vx} l^{(T)}(\hat \vx, \hat \vy)\|_2) \cdot \| \hat \vx  - \vx^*\|_2 .
\end{split}
\end{equation}
Thus, we have
\begin{equation}
    \|\hat \vx - \vx^*\|_2 \leq \frac{1}{\mu} \cdot (\| \nabla l^*(\hat \vx) - \nabla_{\vx} l^{(T)}(\hat \vx, y^*(\hat \vx))\|_2 + \| \nabla_{\vx} l^{(T)}(\hat \vx, y^*(\hat \vx)) - \nabla_{\vx} l^{(T)}(\hat \vx, \hat \vy)\|_2).
\end{equation}
By applying Proposition \ref{prop:key-result}, we then have
\begin{equation}
\begin{split}
     \|\nabla l^*(\hat \vx) - \nabla_{\vx} l^{(T)}(\hat \vx, y^*(\hat \vx)) \|_2
    &\leq 
    \|\nabla y^*(\hat \vx)  - \nabla_{\vx} h^{(T)}(\hat \vx, y^*(\hat \vx)) \|_2 \cdot \|\nabla_{\vy} l(\hat \vx, y^*(\hat \vx))\|_2 \leq \frac{G_y L_x \eta^{T/2}}{\gamma}.
\end{split}
\end{equation}
Meanwhile, using the triangle inequality of norms, we have
\begin{equation}
\begin{split}
    &\| \nabla_{\vx} l^{(T)}(\hat \vx, y^*(\hat \vx)) - \nabla_{\vx} l^{(T)}(\hat \vx, \hat \vy)\|_2 = \|\nabla_{\vx} l(\hat \vx, y^*(\hat \vx)) -
    \nabla_{\vx} l(\hat \vx, h^{(T)}(\hat \vx, \hat \vy)) \|_2 \\
    &\qquad \qquad \qquad \qquad \qquad + \|\nabla_{\vx} h^{(T)}(\hat \vx, y^*(\hat \vx)) \|_2 \cdot \|\nabla_{\vy} l(\hat \vx, y^*(\hat \vx)) - \nabla_{\vy} l(\hat \vx, h^{(T)}(\hat \vx, \hat \vy)\|_2  \\
    &\qquad \qquad \qquad \qquad \qquad + \|\nabla_{\vx} h^{(T)}(\hat \vx, y^*(\hat \vx)) - \nabla_{\vx} h^{(T)}(\hat \vx, \hat \vy)\|_2 \cdot \| \nabla_{\vy} l(\hat \vx, h^{(T)}(\hat \vx, \hat \vy))\|_2.
\end{split}
\end{equation}
Firstly, noted that $\nabla_{\vx} l(\vx, \cdot)$ is $G_{xy}$-Lipschitz continuous, we have
\begin{equation}
    \|\nabla_{\vx} l(\hat \vx, y^*(\hat \vx)) -
    \nabla_{\vx} l(\hat \vx, h^{(T)}(\hat \vx, \hat \vy)) \|_2 \leq G_{xy} \cdot \| y^*(\hat \vx) - h^{(T)}(\hat \vx, \hat \vy) \|_2 \leq G_{xy} d_{\max} \eta^{T/2},
\end{equation}
where the second inequality comes from Equation \eqref{eq:h-convergence-up}.

Secondly, by applying the triangle inequality, we can obtain
\begin{equation}
\begin{split}
    \|\nabla_{\vx} h^{(T)}(\hat \vx, y^*(\hat \vx)) \|_2 &\leq \|\nabla_{\vx} h^{(T)}(\hat \vx, y^*(\hat \vx)) - \nabla_{\vx} y^*(\hat \vx)\|_2 + \|\nabla y^*(\hat \vx) \|_2 \leq \frac{L_x}{\gamma} (1 + \eta^{T/2}).
\end{split}
\end{equation}
where the second inequality comes from Lemma \ref{prop:y-star-lip} and \ref{thm:convergence-strong}. Noting that $\nabla_{\vy} l(\vx, \cdot)$ is $G_{yy}$-Lipschitz continuous, we also have
\begin{equation}
    \|\nabla_{\vy} l(\hat \vx, y^*(\hat \vx)) - \nabla_{\vy} l(\hat \vx, h^{(T)}(\hat \vx, \hat \vy)\|_2 \leq G_{yy} \cdot \|y^*(\hat \vx) -  h^{(T)}(\hat \vx, \hat \vy)\|_2 \leq G_{yy} d_{\max} \eta^{T/2},
\end{equation}
where the second inequality also comes from \eqref{eq:h-convergence-up}.

Thirdly, we have $\| \nabla_{\vy} l(\hat \vx, h^{(T)}(\hat \vx, \hat \vy))\|_2 \leq G_y$. Meanwhile, we also have
\begin{equation}
\begin{split}
     &\|\nabla_{\vx} h^{(T)}(\hat \vx, y^*(\hat \vx)) - \nabla_{\vx} h^{(T)}(\hat \vx, \hat \vy)\|_2 \\
     &\qquad \qquad \qquad \qquad \qquad \leq \|\nabla_{\vx} h^{(T)}(\hat \vx, y^*(\hat \vx)) - \nabla_{\vx} y^*(\hat \vx)\|_2 + \|\nabla_{\vx} h^{(T)}(\hat \vx, \hat \vy) - \nabla_{\vx} y^*(\hat \vx) \|_2 \\
     &\qquad \qquad \qquad \qquad \qquad \leq \frac{L_x \eta^{T/2}}{\gamma} + \left(H_x + T d_{\max} \left( H_{xy}  + \frac{L_x}{\gamma} \cdot H_{yy} \right) \right)\eta^{T / 2},
\end{split}
\end{equation}
where the second inequality comes from Lemma \ref{prop:key-again}.

Combing these, we obtain that
\begin{equation}
\begin{split}
    \mu \|\hat \vx^* - \vx^*\|_2 &\leq \frac{G_y L_x \eta^{T/2}}{\gamma} + G_{xy} d_{\max} \cdot \eta^{T/2} + \frac{L_x}{\gamma} (1 + \eta^{T/2}) G_{yy} d_{\max} \cdot \eta^{T/2} + \\
    &\qquad G_y \left( \frac{L_x \eta^{T/2}}{\gamma} + \left(H_x + T d_{\max} \left( H_{xy}  + \frac{L_x}{\gamma} \cdot H_{yy} \right) \right) \cdot \eta^{T / 2} \right).
\end{split}
\end{equation}
Without loss of generality, we assume that $\gamma = 1$ and $\mu = 1$. Then we have
\begin{equation}
\begin{split}
    \|\hat \vx^* - \vx^*\|_2 &\leq G_y L_x \cdot \eta^{T/2} + G_{xy} d_{\max} \cdot \eta^{T/2} + L_x (1 + \eta^{T/2}) G_{yy} d_{\max} \cdot \eta^{T/2} \\
    &\qquad + G_y \left( L_x \cdot \eta^{T/2} + \left(H_x + T d_{\max} \left( H_{xy}  + L_x \cdot H_{yy} \right) \right) \cdot \eta^{T / 2} \right) \\
    &\leq \left(2G_y L_x +  G_{xy} d_{\max} + 2 L_x G_{yy} d_{\max}  + G_y H_x\right) \cdot \eta^{T / 2} \\
    &\qquad + G_y (H_{xy} + L_x H_yy) \cdot (T + 1) \cdot \eta^{T / 2}.
\end{split}
\end{equation}
Eventually, we obtain that
\begin{equation}
\begin{split}
    l^*(\vx^*) - l(\hat \vx, h^{(T)}(\hat \vx, \hat \vy)) &\leq 
    G_y (H_{xy} + L_x H_yy) \cdot (T + 1) \cdot \eta^{T / 2}
    + M \cdot \eta^{T/2},
\end{split}
\end{equation}
where $M = (G_x + L_x G_y) \cdot \left(2G_y L_x +  G_{xy} d_{\max} + 2 L_x G_{yy} d_{\max} + G_y H_x\right) + G_y d_{\max}$.

\section{More Discussion on the Results in Section \ref{sec:main-results}}

\subsection{Gradient Calculation in Algorithms \ref{alg:cournot} and \ref{alg:monopoly}}
\label{app:dp-gradient}

Below we briefly discuss how to code $h(\vx, \vy)$ so that it can be fed into differentiable-programming frameworks. For simplicity, we denote $g(\vz) = \argmin_{\vy \in \sY} \|\vy - \vz\|_2$, then  $h(\vx, \vy) = g(\vy - r \cdot f(\vx, \vy))$, where $\vz = \vy - r \cdot f(\vx, \vy)$. Thus, the problem reduces to how to code $g(\vz)$. Since $g(\vz)$ is equivalent to a Euclidean projection problem, it can always be coded via \texttt{cvxpylayers} \citep{agrawal2019differentiable}, as discussed earlier. 
Here we consider some other coding strategies.  For example, there are three cases when $g(\vz)$ is analytic, so that we can directly code $g(\vz)$ according to its analytic expression.
\begin{itemize}
    \item (i) When  $\sY = \sR^{n}$ (the full space), we always have $g(\vz) = \vz$.
    \item (ii) When $\sY = \{\vy: \va \leq \vy \leq \vb\}$ (a box constraint), we have $g(\vz) = \min\{\max\{\vz, \va\}, \vb\}$.
    \item (iii) When $\sY = \{\vy: \vone^{\T} \vy = 1\}$ (an equity constraint), we have $g(\vz) = \vz -  (\vone^{\T} \vz - 1) / n$.
\end{itemize}
We are now ready to handle a more complicated case when $\sY$ is probability simplex as in many practical applications. 
\begin{itemize}
    \item (iv) When $\sY = \{\vy \geq \vzero: \vone^{\T} \vy = 1\}$ (a \textit{probability simplex constraint}), we can rewrite $\sY = \{\vy: \vy \geq \vzero\} \cap \{\vy: \vone^{\T} \vy = 1\}$, which is the intersection of a box constraint and an equity constraint. Thus, we can use Dykstra's projection algorithm \citep{boyle1986method} to solve $g(\vz)$, which iteratively projects $\vz$ onto two sets until a fixed point is found. This alternating projection process is differentiable because the projections on both sets respectively correspond to cases (ii) and (iii), both of which are analytic. 
\end{itemize}
Subsequently, we can move to a more complicated case when $\sY$ is rewritten as $\sY = \sY_1 \times \cdots \times \sY_n$, i.e., the Cartesian product of some other sets, the projection then can be carried out on $\sY_i$ ($i = 1, \ldots, n$) independently. Hence, as long as every $\sY_i$ falls into the four cases discussed above, the function  $g(\vz)$ can still be coded as a differentiable program easily. Decomposing the projection problem $g(\vz)$ into $n$ parallel sub-problems can also accelerate AD.

\begin{remark}
Here we mark that a ``differentiable program" is not necessarily everywhere differentiable. Instead, it only means that the program can be fed to differentiable-programming frameworks. For example, in case (ii), the operation $\min\{\max\{\vz, \va\}, \vb\}$ is not differentiable at $\vz = \va$ and $\vz = \vb$; at those points, a sub-gradient can be used to run AD.

\end{remark}

\subsection{Other Functions for Formulating the Two Models}
\label{app:other-h}

To formulate the two proposed models, the selection of $h(\vx, \vy)$ is not absolute. Indeed, we have shown that whenever $h(\vx, \vy)$ charts a provably convergent path to solve the lower-level VI in Problem \eqref{eq:bilevel}, the resulting $T$-step Cournot games and monopoly models will provide an upper and a lower bound to Problem \eqref{eq:bilevel}, respectively. For example, the mirror descent method \citep{nemirovskij1983problem} gives
\begin{equation}
    h(\vx, \vy) = \argmin_{y' \in \sY}~ r \cdot  \langle f(\vx, \vy), \vy' - \vy \rangle + D_{\phi}(\vy', \vy),
    \label{eq:bregman-extension}
\end{equation}
where $D_{\phi}(\vy', \vy) =  \phi(\vy') - \phi(\vy) - \langle \nabla \phi(\vy), \vy' - \vy\rangle$ is the Bregman divergence between $\vy'$ and $\vy$ induced by a strongly convex function $\phi: \sY \to \sR$. We refer the readers to \citet{mertikopoulos2019learning} for sufficient conditions under which the mirror descent algorithm converges to the solution to a VI. When $\sY$ is a probability simplex (or the Cartesian product of several probability simplices) and $D_{\phi}(\vy', \vy)$ is specified as the KL divergence, $h(\vx, \vy)$ becomes analytic \citep{beck2003mirror}, so that it can be more easily differentiated through via AD.

\subsection{Global Optimization Strategies}
\label{app:global}

In Theorem \ref{thm:main}, we assume $l^*(\vx)$ to be globally convex; in general, this assumption does not hold for most practical applications. In this section, we briefly discuss how to search for the global minimizer of Problem \eqref{eq:bilevel}. Typically, when an optimization problem is not convex, then we have no choice but to accept the best local solution after running an appropriate first-order method multiple times with different initial solutions. But for bilevel programs, repeatedly searching for local solutions is not easy. Nevertheless, with our framework, we may first fix $T$ as a small value (e.g., 0 or 1) and first run Algorithm \ref{alg:cournot} and/or Algorithm \ref{alg:monopoly} many times, each time with a randomly generated initial solution. This step may be viewed as a quick-and-dirty ``scan" of the geometry of the overall problem. Afterward, the best local solution to either $T$-step Cournot games or monopoly models then can be used to roughly estimate where the global optimal solution to \eqref{eq:bilevel} might be located. Subsequently, we can search for the optimal solution to Problem \eqref{eq:bilevel} around that region, following the procedure that we devised at the end of Section \ref{sec:gap}.

\subsection{Extension to Bilevel Programs with Multiple Lower-Level Solutions}
\label{sec:non-unique}

We then discuss the case when $\sY^*(\vx)$ is not a singleton, which renders a principled rule for selecting a $\vy^* \in \sY^*(\vx)$ to evaluate $l(\vx, \vy^*)$ as a necessity. To this end, the optimistic and pessimistic principles respectively select  the best and worst solutions, as judged by the leader's objective. In other words, they respectively solve
\begin{equation}
    \min_{\vx \in \sX} \ \min_{\vy^* \in \sY^*(\vx)}~l(\vx, \vy^*) \quad \text{and} \quad \min_{\vx \in \sX} \ \max_{\vy^* \in \sY^*(\vx)}~l(\vx, \vy^*).
\end{equation}
The above two problems are commonly referred to as a \textit{strong} and a \textit{weak} Stackelberg game respectively \citep{loridan1996weak}.
Denoting the solutions to those two problems as $(\vx_{\text{opt}}^*, \vy_{\text{opt}}^*)$ and $(\vx_{\text{pes}}^*, \vy_{\text{pes}}^*)$, then we have
$
    l(\vx_{\text{opt}}^*, \vy_{\text{opt}}^*) \leq l(\vx_{\text{pes}}^*, \vy_{\text{pes}}^*).
$
In the literature, the strong Stackelberg games is more frequently regarded as the ``standard" formulation \citep{luo1996mathematical}. The original formulation \eqref{eq:bilevel} is indeed a strong Stackelberg game because it allows the leader to pick any $\vy^* \in \sY^*(\vx)$. Thus, the leader can always pick the most favorable (i.e., the optimistic) one. According to Proposition \ref{prop:model-gap}, the solutions given by $T$-step Cournot games and monopoly models always provide an upper and a lower bound for $l(\vx_{\text{opt}}^*, \vy_{\text{opt}}^*)$, respectively. We are hence motivated to ask whether our models would still provide \textit{arbitrarily tight} bounds to $l(\vx_{\text{opt}}^*, \vy_{\text{opt}}^*)$ when $T \to \infty$.

Unfortunately, we cannot answer this question with a certain yes. Indeed, in a $T$-step Cournot game, the followers' decision $\bar \vy$ is not necessarily the one that minimizes $l(\bar \vx, \vy^*)$ among all $\vy^* \in \sY^*(\bar \vx)$. In other words, a $T$-step Cournot game is not intrinsically optimistic; instead, it has to let the followers pick an equilibrium by themselves. Hence, even though $T$ is sufficiently large, a $T$-step Cournot game may still admit multiple solutions. If one solution $(\bar \vx, \bar \vy)$ accidentally satisfies $\bar \vy \in \argmin_{\vy^*} l(\bar \vx, \vy^*)$, then we expect $l(\bar \vx, \bar \vy)$ may be close to $l(\vx_{\text{opt}}^*, \vy_{\text{opt}}^*)$. Otherwise, if $\bar \vy$ is the solution in $\sY^*(\bar \vx)$ against the leader's interest, the gap may still  be unacceptably large. Below we give an example.

\begin{example}
Consider the following problem
\begin{equation}
\begin{split}
    \min_{1 \geq x \geq 0.5}~~&(x + y^* - 1)^2 \\
    \text{s.t.}~~&y^* \in \argmin_{y \in \sR}~g(x, y) = 
    \begin{cases}
    (y - x - 0.25)^2, \quad \text{if}~y \geq x + 0.25, \\
    \ 0, \quad \text{if}~y \in (x - 0.25, x + 0.25), \\
    (y - x + 0.25)^2, \quad \text{if}~y \leq x - 0.25,
    \end{cases}
\end{split}
\end{equation}
which is equivalent to the following one
\begin{equation}
\begin{split}
    \min_{1 \geq x \geq 0.5}~~&(x + y^* - 1)^2 \\
    \text{s.t.}~~&y^* \in [x - 0.25, x + 0.25].
    \label{eq:example}
\end{split}
\end{equation}
Under this setting, the optimal solution set to \eqref{eq:example} is $\sZ^* = \{(x, y): x \in [0.5, 1], y = 1 - x\}$ and the solution set to the corresponding $0$-step Cournot game is $\sZ^{(0)} = \sZ^* \cup \{(x, y): x = 0.5, y \in [0.5, 0.75]\}$.
We then prove that the solution sets to $T$-step Cournot models are still $\sZ^{(0)}$ for alll $T > 0$. 
\begin{equation}
    \nabla_y g(x, y) = 
    \begin{cases}
    2(y - x - 0.25), \quad \text{if}~y \geq x + 0.25, \\
    \ 0, \quad \text{if}~y \in (x - 0.25, x + 0.25), \\
    2(y - x + 0.25), \quad \text{if}~y \leq x - 0.25.
    \end{cases}
\end{equation}
We can then define
\begin{equation}
    h(x, y) = \argmin_{1 \geq x \geq 0.5}~~y - r \cdot \nabla_y g(x, y)
\end{equation}
to model the dynamics, which satisfies $h(x, y) = y$ for any $x \in [0.5, 1]$ and $y \in [x - 0.25, x + 0.25]$. Thus, for any $\bar y \in [0.5, 0.75]$, the optimal solution to 
\begin{equation}
    \min_{1 \geq x \geq 0.5}~~(x + h^{(T)}(x, \bar y) - 1)^2 = (x + \bar y - 1)^2
\end{equation}
is still $\bar x = 0.5$ for all $T > 0$. Therefore, we have $\sZ^{(T)} = \sZ^* \cup \{(x, y): x = 0.5, y \in [0.5, 0.75]\}$ for all $T \in \sN$. Among these solutions, the worst one is $(\bar x, \bar y) = (0.5, 0.75)$, whose corresponding objective value is $0.0625 > 0$.
\end{example}

In contrast, a $T$-step monopoly model, to some extent, is essentially optimistic. Intuitively, if $(\hat \vx, \hat \vy)$ solves a $T$-step monopoly model with a sufficiently large $T$, we may expect that $h^{(T)}(\hat \vx, \hat \vy)$ is close to the one that minimizes $l(\hat \vx, \vy^*)$ among all $\vy^* \in \sY^*(\hat \vx)$, as the leader has the authority to steer the follower's evolutionary path by manipulating their initial strategy.

Motivated by the above discussion, we design a heuristic for approximating Problem \eqref{eq:bilevel} by iteratively solving $T$-step Cournot games and monopoly models. Specifically, setting $T$ as a small number and starting with an arbitrary $(\vx^0, \vy^0) \in \sX \times \sY$, we may first run run Algorithms \ref{alg:cournot} and \ref{alg:monopoly} to obtain $(\bar \vx, \bar \vy)$ and $(\hat \vx, \hat \vy)$, respectively. If the gap between $l^{(T)}(\bar \vx, \bar \vy)$ and $l^{(T)}(\hat \vx, \hat \vy)$ is sufficiently small, then we can accept $(\bar \vx, \bar \vy)$ as an approximated solution to Problem \eqref{eq:bilevel}. Otherwise, we can increase $T$ and repeat the above procedure. But during the next iteration, we may use a warm-start initial solution, rather than the original $(\vx^0, \vy^0)$, to run Algorithms \ref{alg:cournot} and \ref{alg:monopoly} to reduce the number of iterations required for convergence. Particularly, that initial solution could set as $ (\hat \vx, h^{(T)}(\hat \vx, \hat \vy))$, where $(\hat \vx, \hat \vy)$ is the output of Algorithm \ref{alg:monopoly} in the previous iteration. As discussed earlier, the pair  $(\hat \vx, h^{(T)}(\hat \vx, \hat \vy))$ is intrinsically optimistic. Hence, if in the next iteration, we run Algorithm \ref{alg:cournot} starting from $(\hat \vx, h^{(T)}(\hat \vx, \hat \vy))$, the new output $(\bar \vx, \bar \vy)$ is more likely to within the neighborhood of $(\vx_{\text{opt}}^*, \vy_{\text{opt}}^*)$. The pseudocode code is provided in Algorithm \ref{alg:adaptive-initialization}. We are not sure whether the gap $\delta^{(T)}$ in Algorithm \ref{alg:adaptive-initialization} would eventually converge to 0 or a sufficiently small value for all practical purposes; we leave a theoretical analysis of the algorithm to future work. But later, we will numerically test it.

\begin{algorithm}[ht]
   \caption{ {\small Adaptive-initialization strategy for approximating Problem \eqref{eq:bilevel}.}}
   \label{alg:adaptive-initialization}
\begin{algorithmic}[1]
{\small
  \STATE {\bfseries Input:} $ (\vx^0, \vy^0) \in \sX \times \sY$, $T$ as a small integer
  \WHILE{True}
    \STATE Starting from $(\vx^0,  \vy^0)$, run Algorithm \ref{alg:cournot} to obtain $(\bar \vx, \bar \vy)$ and then run Algorithm \ref{alg:monopoly} to obtain  $(\hat \vx, \hat \vy)$.
    \STATE If $l^{(T)}(\bar \vx, \bar \vy) - l^{(T)}(\hat \vx, \hat \vy)$ is sufficiently small, break and return $(\bar \vx, \bar \vy)$; otherwise, increase $T$ and set $(\vx^0,  \vy^0) = (\hat \vx, h^{(T)}(\hat \vx, \hat \vy))$.
  \ENDWHILE
  }
\end{algorithmic}
\end{algorithm}

\section{Numerical Experiments}
\label{app:experiments}

In Experiment I (Appendix \ref{app:exp-1}), we will study a bilevel program whose lower-level problem does not admit a unique solution. Algorithm \ref{alg:adaptive-initialization} will be numerically tested. In Experiment II (Appendix \ref{app:exp-2}), we study the network design problem on the Braess network \citep{braess1968paradoxon}. We will compare the solutions provided by the $T$-step Cournot games and monopoly models formulated by two types of $h(\vx, \vy)$: the projection method and the mirror descent method (see Appendix \ref{app:other-h}). Eventually, in Experiment III (Appendix \ref{app:exp-3}), we study the network problem on a real transportation network and compare the proposed scheme with many existing bilevel programming algorithms.

\subsection{Experiment I}
\label{app:exp-1}

We first consider a bilevel program of the following form
\begin{equation}
    \begin{split}
        \small \min_{\va + \vx \geq 0}~&\|\vx\|_2 + \langle \mD^{\T} (\va + \vx + \vb \cdot \vv), \vw \cdot \vy^* \rangle, \\
         \text{s.t.}~&\evx_4 = 0, \quad \vv = \mD \vy, \\
        &\vy^* \in \argmin_{\vy \geq 0, \ \textbf{1}^{\T} \vy = 1}~ \langle \textbf{1}^{\T}, (\va + \vx) \cdot \vv + \frac{1}{2} \cdot \vb \cdot \vv^2 \rangle,
    \end{split}
\end{equation}
where $\vx \in \sR^4$, $\vy \in \sR^4$, $\vw = [2, 1.1, 0.9, 0.01]^{\T}$, $\vb = [15, 2, 8, 5]^{\T}$, $\va = [2, 3, 4, 5]^{\T}$, $\mD = [\ve_1 + \ve_3, \ve_2 + \ve_4, \ve_1 + \ve_4, \ve_2 + \ve_3]$. Thus, the lower level is
a convex but not strongly convex optimization, whose optimal solution set is a non-singleton. We first run Algorithm \ref{alg:cournot} and \ref{alg:monopoly} for $T = 0, \ldots, 5$. At each round, $\vx^0$ is randomly sampled from a Gaussian distribution while $\vy^0$ is first randomly sampled from a uniform distribution and then re-weighted to fit the constraint. The result (see Figure \ref{fig:experiment2-b}) shows the boxplot of the objective values reached from different initial points. It verifies our conclusions that $T$-step Cournot games may have multiple solutions with different objective values, while $T$-step monopoly models always have a unique optimal objective value.

\begin{figure}[ht]
\begin{center}
\centerline{\includegraphics[width=0.6\columnwidth]{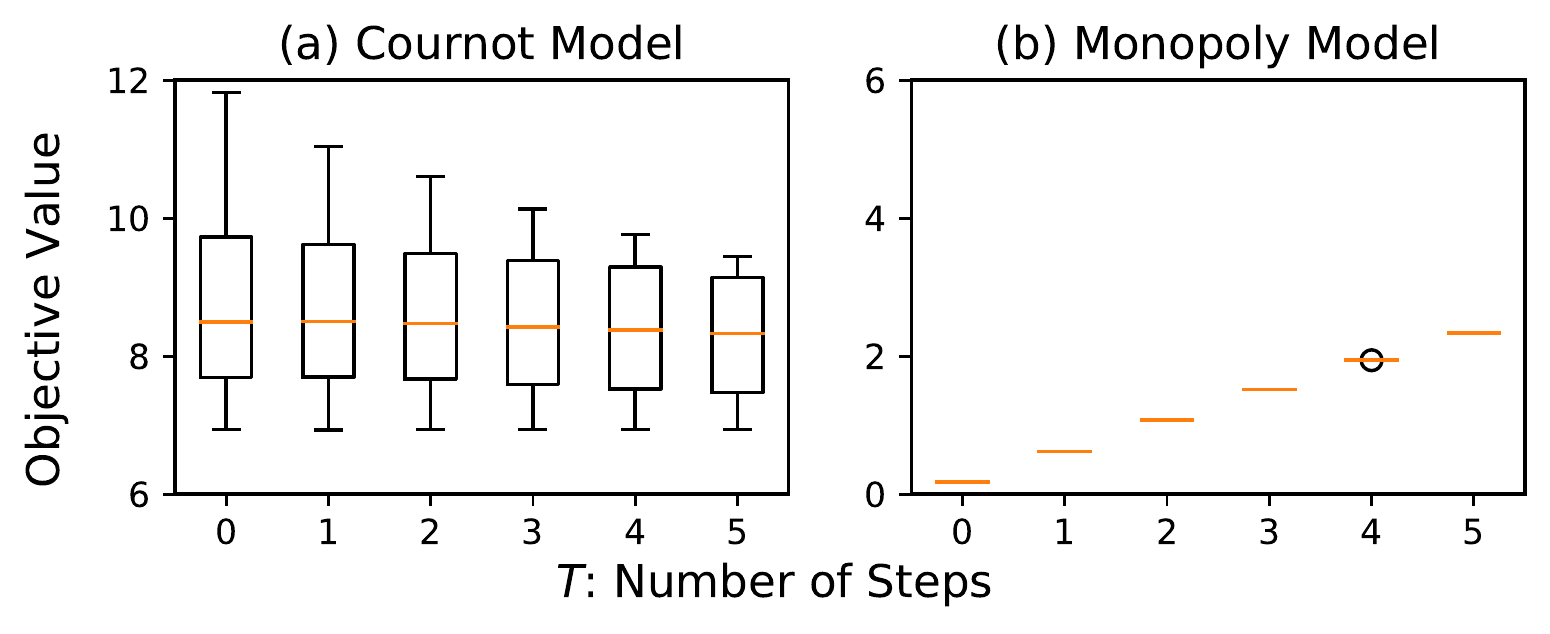}}
\vskip -0.1in
\caption{Comparison between Algorithm \ref{alg:cournot} and \ref{alg:monopoly} starting from different initial points.}
\label{fig:experiment2-b}
\end{center}
\vskip -0.2in
\end{figure}

We then compare two initialization strategies. (a) Fixed initialization strategy: we follow the procedure devised at the end of Section \ref{sec:algorithm}; the initial solutions are set as $\vx^0 = [0, 0, 0, 0]^{\T}$ and $\vy^0 = [0.4, 0.3, 0.2, 0.1]$ whenever Algorithms \ref{alg:cournot} and \ref{alg:monopoly} are invoked. (b) Adaptive initialization strategy: we directly run Algorithm \ref{alg:adaptive-initialization} with $\vx^0 = [0, 0, 0, 0]^{\T}$ and $\vy^0 = [0.4, 0.3, 0.2, 0.1]$ as the initial input. When testing both strategies, we gradually increase $T$ as $0, 1, 2, 3, 4, 5, 7, 10, 20, 30, 40, 50, 60, 70$. The result (see Figure \ref{fig:experiment2-c}) shows that when using the fixed-initialization strategy, the gap between upper and lower bounds is still non-eligible when $T = 70$. On the contrary, when using the adaptive-initialization strategy, a near-optimal feasible is found when $T$ is only $1$ because the previous monopoly model identifies a neighborhood of the ``optimistic" solution.  Meanwhile, the gap between the two models gradually decreases to 0, eventually settling down at the exact optimal solution.

\begin{figure}[ht]
\begin{center}
\centerline{\includegraphics[width=0.6\columnwidth]{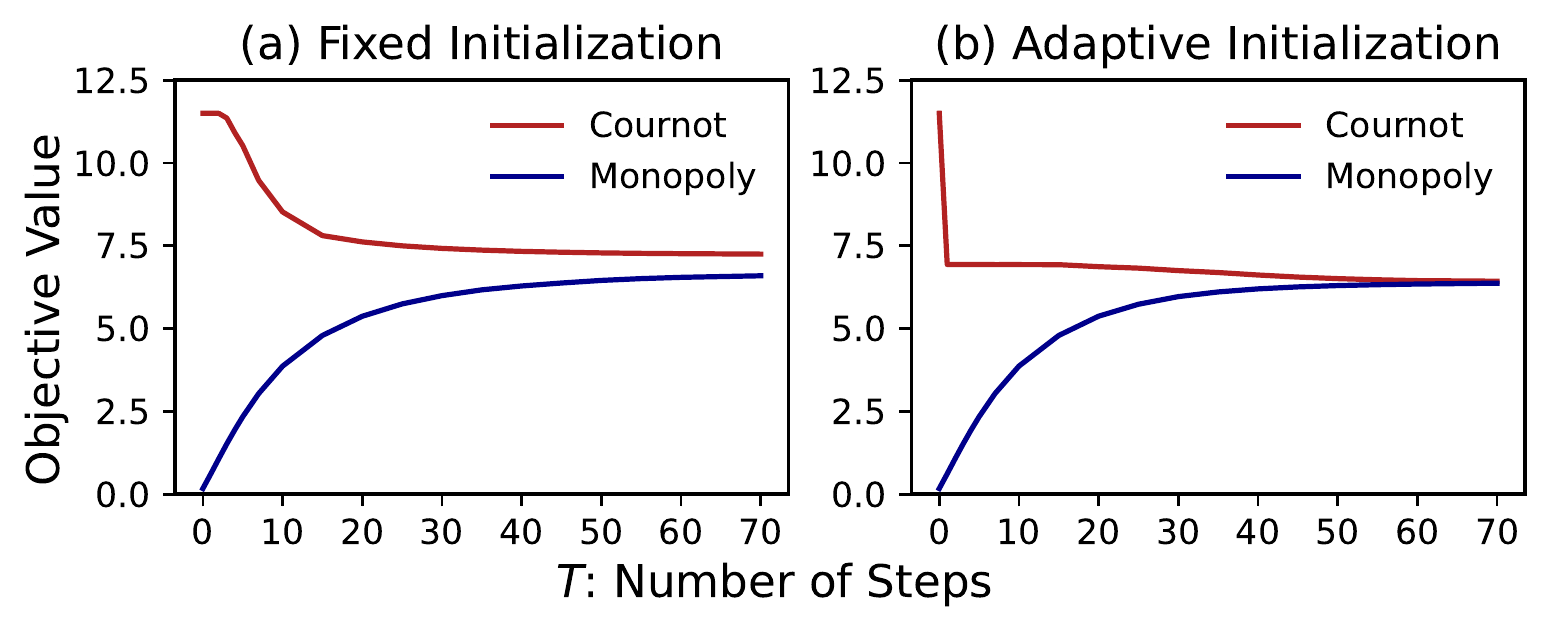}}
\vskip -0.1in
\caption{Comparison between fixed- and adaptive-initialization strategies (Algorithm \ref{alg:adaptive-initialization}).}
\label{fig:experiment2-c}
\end{center}
\vskip -0.2in
\end{figure}

In the future, we will test Algorithm \ref{alg:adaptive-initialization} on more complicated examples. But in this special example, its effectiveness in handling bilevel programs with non-unique lower-level solutions is validated.

\subsection{Experiment II}
\label{app:exp-2}

We then test our algorithms on a classic network design problem, which studies how to add capacities in a traffic network to reduce congestion. We model the network as a directed graph $\gG(\sN, \sA)$, where $\sN$ and $\sA$ are the set of nodes and arcs. We write $\sW \subseteq \sN \times \sN$ as the set of origin-destination (OD) pairs and $\sK \subseteq 2^{\sA}$ as the set of paths connecting the OD pairs.  We assume that each OD pair $w \in \sW$ is associated with $\evq_w$ vehicles. Let $\sK_w \subseteq \sK$ be the set of all paths connecting $w$. We assume that the traffic planner's objective is a weighted sum of the total monetary cost associated with the capacity enhancement and the total travel time experienced by the vehicles. We denote the original arc capacity as $\vs \in \sR_{+}^{|\sA|}$ and the capacity enhancement added by the planning agent as $\vx$. We assume that the capacities can only be added to selected arcs, denoted as $\tilde \sA \subseteq \sA$. The feasible region for $\vx$ then becomes $\sX = \{\vx \in \sR_{+}^{|\sA|}: \evx_a = 0, a \in \sA \setminus \tilde \sA\}$. We write $m(\vx) = \langle \vb, \vx^2 \rangle$ as the total monetary cost associated with the capacity enhancement, where $\vb \in \sR_{+}^{|\sA|}$ are cost parameters.
We assume that the arc travel time is given by $u(\vx, \vv) = \vu_0 \cdot \left(1 + 0.15 \cdot (\vv / (\vs + \vx)) ^ 4 \right)$, where $\vv, \vu_0 \in \sR_{+}^{|\sA|}$ represent arc vehicle flows and free-flow travel times, respectively.  Meanwhile, we write the vehicles' route choices as a vector $\vy = (\evy_k)_{k \in \sK}$ with $\evy_{k}$ equals the proportion of vehicles between  $w$ using the path $k \in \sK_w$, and the travel times for using each path as $\vc = (\evc_k)_{k \in \sK}$. Denote $\emSigma_{w,k}$ as the OD-path incidence, with $\emSigma_{w,k}$ equals 1 if the path $k \in \sK_w$ and 0 otherwise. Meanwhile, denote $\emLambda_{a,k}$ as the arc-path incidence, with $\emLambda_{a,k}$ equals 1 if $a \in \sA_k$ and 0 otherwise. For notational convenience, we write $\mLambda = (\emLambda_{a,k})_{a \in \sA, k \in \sK}$ and $\mSigma = (\emSigma_{w,k})_{w \in \sW, k \in \sK}$. Let $\vd = (\evd_k)_{k \in \sK}$ be a vector with $\evd_k = \evq_w$ if $k \in \sK_w$. The feasible region for $\vy$ can then be written as
$\sY = \{\vy \geq 0: \mSigma \, \vy = \bm 1\}.$ Meanwhile, we also have $\vv = \mLambda (\vd \cdot \vy)$ and $\vc = \mLambda^{\T} u(\vx, \vv)$. The results in Example \ref{eg:wardrop} imply that the set of Wardrop equilibria $\sY^*(\vx)$ is the solution set to the following VI problem:
\begin{equation}
    \langle \mLambda^{\T} u(\vx, \mLambda (\vd \cdot \vy^*)), \vy - \vy^* \rangle \geq 0, \quad \forall \vy \in \sY.
\end{equation}
The network design problem then has the following form 
\begin{equation}
\begin{split}
    \min_{\vx \in \sX}~~ &\langle u(\vx, \vv^*)), \vv^*) \rangle + \gamma \cdot m(\vx), \\
    \text{s.t.}~~&\vv^* = \mLambda (\vd \cdot \vy^*), \quad \vy^* \in \sY^*(\vx).
\end{split}
\label{eq:capacity-design}
\end{equation}

We first solve the network design problem on the Braess network \citep{braess1968paradoxon} as shown in Figure \ref{fig:n1}. The network has three paths connecting the origin (node 1) and the destination (node 4): path 1 uses links 1 and 3, path 2 uses links 1, 4, and 5, and path 3 uses links 2 and 5. The Braess paradox \citep{braess1968paradoxon} implies that no capacities should be added to link 4 (the bridge link). Otherwise, it would increase the total travel time experienced by the travelers at equilibrium.
\begin{figure}[H]
\centering
\includegraphics[width=0.28\textwidth]{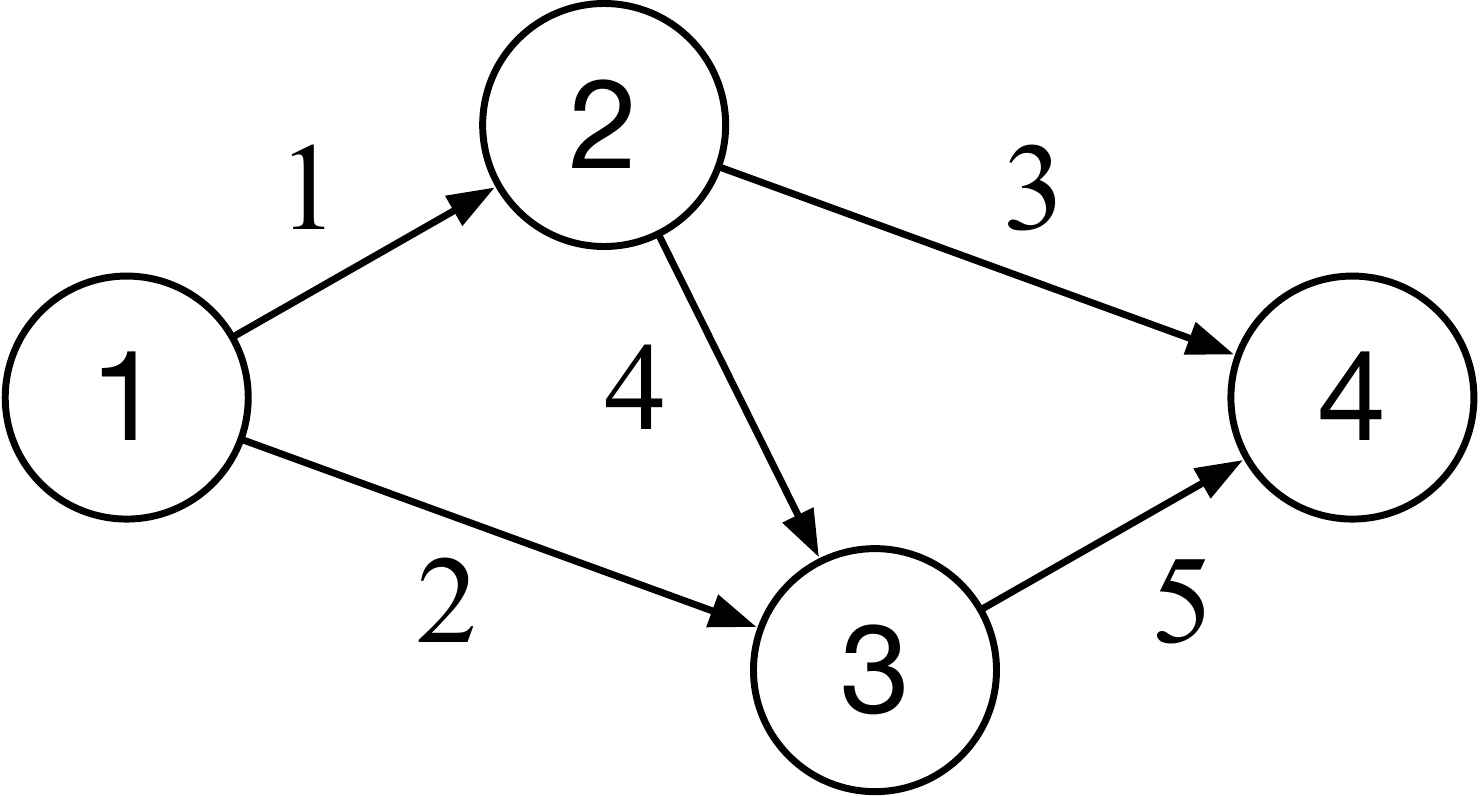}
    \captionof{figure}{The Braess network.}
    \label{fig:n1}
\end{figure}

We set $m(\vx) = \langle \vw, \vx^2 \rangle$, $\vd = 6$, $\vu_0 = [1, 3, 3, 0.5, 1]^{\T}$, $\vv_0 = [2, 4, 4, 1, 2]^{\T}$, $\vw = [1, 3, 3, 0.5, 1]^{\T}$, and $\gamma = 1$. The problem is first solved via two existing AD-based methods proposed by \citet{li2020end} and \citet{li2022differentiable}, both of which unroll the full computational process for solving the lower-level VI problem. The difference lies in the algorithm being unrolled: the first unrolls the projection method while the second unrolls the mirror descent method (see Section \ref{app:other-h}). In the following experiments, the solutions returned by \citet{li2020end}'s and \citet{li2022differentiable}'s methods will be used as the benchmark. When testing our scheme, we also formulate $T$-step Cournot games and monopoly models with two types of $h(\vx, \vy)$: the projection method and the mirror descent method. To investigate whether their solutions can make a nice approximation to the benchmark solutions,  we progressively increase $T$ in the experiment 
until the gap is closed. Our models are solved by Algorithms \ref{alg:cournot} and \ref{alg:monopoly}; the same hyperparameters are employed for all tested algorithms (including both the benchmarks and our algorithms), except the learning rate $r$, which is set to be $0.1$ in the projection-method version and $0.25$ in the mirror-descent version. Table \ref{tab:exp-1-2}  reports the solutions  (upper-level: capacity enhancement; lower-level: route choice), the corresponding objective function values as well as the total CPU (2.9 GHz Quad-Core Intel Core i7) time and the number of iterations required to obtain the solutions (for notational simplicity, we use ``S" to represent the benchmarks and use ``C-" and ``M-" to represent $T$-step Cournot games and monopoly models with different $T$s). 
\begin{table}[htb]
\centering
\caption{Performance of $T$-step Cournot games and $T$-step monopoly models when solving the continuous network design problem on the Braess network.}
\label{tab:exp-1-2}
\vspace{0pt}
\begin{subtable}{0.99\textwidth}
  \footnotesize
  \centering
  \caption{Setting A: Modeling the lower-level solution process using the projection method.}
  \vspace{-2pt}
    \begin{tabular}{cccccccccccccc}
    \toprule
    \multirow{2}[4]{*}{Method} & \multirow{2}[4]{*}{Time (s)} & \multirow{2}[4]{*}{Iterations} & \multirow{2}[4]{*}{Value} &       & \multicolumn{5}{c}{Capacity enhancement} &       & \multicolumn{3}{c}{Route choice} \\
    \cmidrule{6-10}\cmidrule{12-14}          &       &       &       &       & $a = 1$ & $a = 2$ & $a = 3$ & $a = 4$ & $a = 5$ &       & $k = 1$ & $k = 2$ & $k = 3$ \\
    \midrule
    C-0   & 1.06  & 80    & 38.786  &       & 2.075  & 0     & 0     & 2.826  & 2.075  &       & 0.000  & 1.000     & 0.000  \\
    C-1   & 6.68  & 109   & 28.920  &       & 0.936  & 0.016  & 0.016  & 0     & 0.936  &       & 0.339  & 0.321  & 0.339  \\
    \midrule
    S     & 12.5  & 92    & 28.920  &       & 0.928  & 0.016  & 0.016  & 0     & 0.928  &       & 0.333  & 0.333  & 0.333  \\
    \midrule
    M-4   & 8.22  & 46    & 28.920  &       & 0.928  & 0.016  & 0.016  & 0     & 0.928  &       & 0.329  & 0.343  & 0.329  \\
    M-3   & 35.1  & 339   & 26.745  &       & 0.867  & 0.027  & 0.027  & 0.192  & 0.867  &       & 0.156  & 0.688  & 0.156  \\
    M-2   & 2.16  & 43    & 26.789  &       & 0.744  & 0.039  & 0.039  & 0.009  & 0.744  &       & 0.217  & 0.565  & 0.217  \\
    M-1   & 0.31  & 40    & 27.142  &       & 0.603  & 0.066  & 0.066  & 0.000  & 0.603  &       & 0.225  & 0.549  & 0.225  \\
    M-0   & 0.03  & 70    & 26.722  &       & 0.821  & 0.030  & 0.030  & 0.113  & 0.821  &       & 0.424  & 0.151  & 0.424  \\
    \bottomrule
    \end{tabular}%
  \label{tab:exp-1-2-1}%
  \end{subtable}

    \vspace{10pt}
    \begin{subtable}{0.99\textwidth}
    \footnotesize
    \centering
    \caption{Setting B: Modeling the lower-level solution process using the mirror descent method.}
    \vspace{-2pt}
    \begin{tabular}{cccccccccccccc}
    \toprule
    \multirow{2}[4]{*}{Method} & \multirow{2}[4]{*}{Time (s)} & \multirow{2}[4]{*}{Iterations} & \multirow{2}[4]{*}{Value} &       & \multicolumn{5}{c}{Capacity enhancement} &       & \multicolumn{3}{c}{Route choice} \\
    \cmidrule{6-10}\cmidrule{12-14}          &       &       &       &       & $a = 1$ & $a = 2$ & $a = 3$ & $a = 4$ & $a = 5$ &       & $k = 1$ & $k = 2$ & $k = 3$ \\
    \midrule
    C-0   & 0.103 & 264   & 38.786 &       & 2.075 & 0     & 0     & 2.830  & 2.075 &       & 0.000     & 1.000     & 0.000 \\
    C-1   & 0.148 & 254   & 28.925 &       & 0.966 & 0.015 & 0.015 & 0     & 0.966 &       & 0.339 & 0.322 & 0.339 \\
    C-2   & 0.208 & 270   & 28.920 &       & 0.939 & 0.016 & 0.016 & 0     & 0.939 &       & 0.339 & 0.321 & 0.339 \\
    \midrule
    S     & 0.393 & 232   & 28.920 &       & 0.928 & 0.016 & 0.016 & 0     & 0.928 &       & 0.333 & 0.333 & 0.333 \\
    \midrule
    M-5   & 0.292 & 232   & 28.920 &       & 0.928 & 0.016 & 0.016 & 0     & 0.928 &       & 0.319 & 0.363 & 0.319 \\
    M-4   & 0.562 & 512   & 26.722 &       & 0.821 & 0.03  & 0.03  & 0.083 & 0.821 &       & 0.169 & 0.663 & 0.169 \\
    M-3   & 0.459 & 440   & 26.722 &       & 0.821 & 0.03  & 0.03  & 0.084 & 0.821 &       & 0.179 & 0.642 & 0.179 \\
    M-2   & 0.323 & 381   & 26.722 &       & 0.821 & 0.03  & 0.03  & 0.085 & 0.821 &       & 0.190 & 0.620 & 0.190 \\
    M-1   & 0.235 & 341   & 26.722 &       & 0.82  & 0.03  & 0.03  & 0.088 & 0.82  &       & 0.202 & 0.596 & 0.202 \\
    M-0   & 0.151 & 447   & 26.722 &       & 0.821 & 0.03  & 0.03  & 0.095 & 0.821 &       & 0.425 & 0.151 & 0.425 \\
    \bottomrule
    \end{tabular}%
  \label{tab:exp-1-2-2}%
  \end{subtable}
\end{table}%

The solutions returned by the two benchmark algorithms are identical;  no capacity is added on link 4 as indicated by the Braess paradox. However, if we adopt $T$-step Cournot games and monopoly models to approximate the original problem, then some capacity would be added on link 4 when $T$ is overly too small, which falls into the trap. However, a slightly larger $T$ will fix the problem in both models, no matter which method is used to model the lower-level solution process.  %
It is worth noting that when $T$ is the same, the solutions to the models formulated by the projection method have better quality.  However, the slight improvement in solution quality can hardly offset the increase in computational cost. Take C-1 for instance. Using the projection method instead of the mirror descent method could decrease the objective value by merely $(28.92 - 28.925) / 28.925 = 0.170\%$, but the required CPU time is increased by $(6.68 - 0.148) / 0.148 = 4410\%$.  Eventually, we remark that a classic result is that C-0 and M-0 can deliver satisfactory solutions on many networks \citep{marcotte1986network}; here the Braess network is intentionally selected as a counterexample. Our goal is to show that if our models are able to make good approximations on the Braess network, then we can be more confident in applying it to other problems.

\subsection{Experiment III}
\label{app:exp-3}

We then move to the network in the City of Sioux Fall, South Dakota; the network data (topology, travel demand, arc travel time function) are downloaded from the \textit{Transportation Network} GitHub repository\footnote{\url{https://github.com/bstabler/TransportationNetworks/tree/master/SiouxFalls}}. For this network, we have $|\sN| = 24$, $|\sA| = 76$, $|\sK| = 638$. We select 10 arcs for expanding the capacities (see Table \ref{tab:arc}). 
\begin{table}[htbp]
  \centering
  \caption{Selected arcs for adding capacities and their corresponding cost parameters.}
  \vskip 0.1in
 \begin{small}
    \begin{tabular}{ccccccccccc}
    \toprule
    $a$     & 16    & 19    & 17    & 20    & 25    & 26    & 29    & 48    & 39    & 74 \\
    $\evb_a$    & 26.0    & 26.0    & 40.0    & 40.0    & 25.0     & 25.0     & 48.0     & 48.0     & 34.0     & 34.0  \\
    \bottomrule
    \end{tabular}%
  \label{tab:arc}%
 \end{small}
\end{table}%

In the experiment, we set $\gamma = 0.01$.  Based on the results given by Experiment II (see Appendix \ref{app:exp-2}), we use the mirror descent method to formulate the proposed models. %
We compare our approaches with some previous bilevel programming methods studied in the optimization and ML literature. 
\begin{itemize}
    \item ``c0": Algorithm \ref{alg:cournot} with $T = 0$. It solves a 0-step Cournot game. The classic single-level approximation scheme proposed by \citep{tan1979hybrid} essentially solves the same model but uses different algorithms.
    \item ``c1": Algorithm \ref{alg:cournot} with $T = 1$. It solves a 1-step Cournot game. It may be viewed as a straightforward extension of one-stage AD \citep{liu2018darts}. 
    \item ``c10": Algorithm \ref{alg:cournot} with $T = 10$. It solves a 10-step Cournot game.
    \item ``m0": Algorithm \ref{alg:monopoly} with $T = 0$. It solves a 0-step monopoly model. The classic single-level approximation scheme proposed by \citep{dantzig1979formulating} essentially solves the same model. 
    \item ``m45": Algorithm \ref{alg:monopoly} with $T = 45$. It solves a 45-step monopoly model.
    \item ``ad": The Algorithm proposed by \citet{li2022differentiable}.  
    \item ``tad": An extension of truncated AD \citep{shaban2019truncated}. It is similar to ``ad"; the difference is that we only unroll the last 10 iterations of the dynamical process for solving the lower-level problem. 
    \item ``id": The implicit differentiation scheme proposed by \citet{li2020end}. 
    \item ``aid": An extension of approximated ID. It shares the same overall structure as ``id"; the difference is that we only keep the first 10 terms in the Neumann series for matrix inversion. 
    \item ``sid": We extend the two-timescale single-looped method proposed by \citet{hong2020two}.
\end{itemize}
We appropriately design hyperparameters in all of these algorithms and stop running the algorithms based on similar termination conditions. The results are shown in Figure \ref{fig:experiment3-2}, in which we report the total CPU (2.9 GHz Quad-Core Intel Core i7) time, the final optimality gaps, the total iteration number as well as the CPU time per iteration.

\begin{figure}[ht]
\begin{center}
\centerline{\includegraphics[width=0.7\columnwidth]{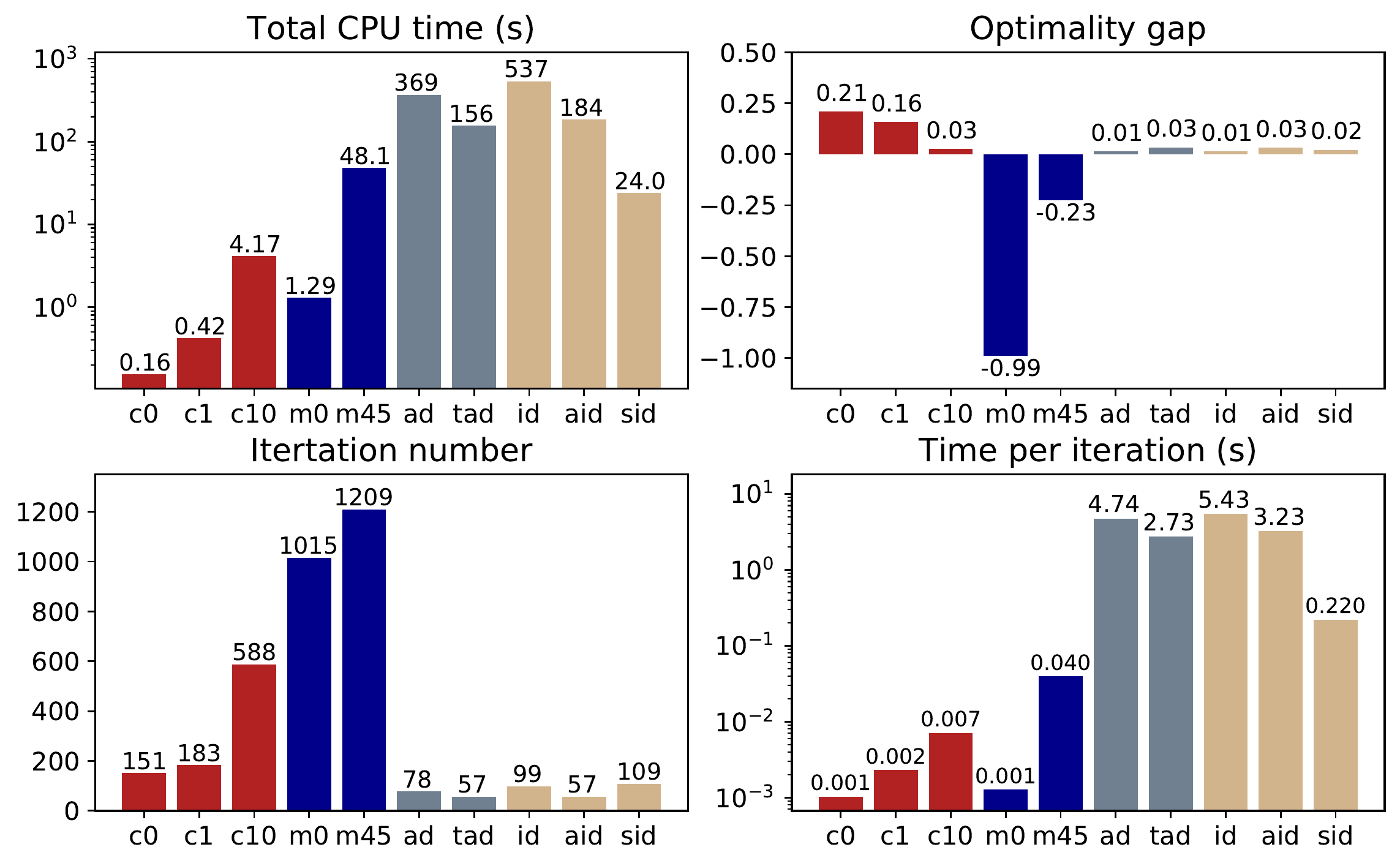}}
\caption{Comparison between our methods and previous methods.}
\label{fig:experiment3-2}
\end{center}
\vskip -0.2in
\end{figure}

\textbf{Observation 1.} We first note that ``c0" and \citet{tan1979hybrid} essentially solve the same models. Compared with the previous approach, our methods can provide more accurate solutions (see ``c1" and ``c10"). Particularly, the optimality gap induced by ``c10" is almost the same as ``tad" and ``aid", the two approximation schemes proposed in the machine learning literature. Meanwhile, it is just slightly larger than the two exact methods, namely, ``id" and ``ad".

\textbf{Observation 2.}  The total CPU time required by ``c10" is significantly lower than ``ad", ``tad", ``id" and ``aid". Specifically, the total number of iterations increases by 5-6 times, but the CPU time per iteration is reduced from 2.5-5.5s to just 0.007s (roughly 400 times). Hence, ``c10" is more efficient.

\textbf{Observation 3.} The single-looped scheme ``sid" is more special. Its CPU time per iteration is obviously much lower than ``ad", ``tad", ``id" and ``aid" because it also bypasses the difficulty in repeatedly solving for the lower-level solution. However, it is still 30 times higher than our approach, mainly because it still builds on implicit differentiation, which involves storing and inverting large matrices. 

\textbf{Observation 4.}  The lower bound provided by ``m45" is not that accurate compared with the upper bound provided by ``c10". Nevertheless, it answers how good the upper bound is. Specifically, ``(upper bound - lower bound) / lower bound" would be a reasonable estimation of the accuracy. If it is smaller than the tolerance, then we are sure that we already find a sufficiently good solution. To the best of our knowledge, no previous method in the machine learning literature can provide such a lower bound for bilevel programs. The classic scheme  \citep{dantzig1979formulating} in the optimization literature could provide such a bound (see ``m0"). However, the gap between the accurate solution and this lower bound is too large. If we use this lower bound to evaluate the upper bound, it may be over too pessimistic.

\bibliographystyle{ims}
\begin{small}
\bibliography{example_paper}
\end{small}

\end{document}

%% file: math_commands.tex
\usepackage{amsmath,amsfonts,bm}

\def\1{\bm{1}}

\def\vzero{{\bm{0}}}
\def\vone{{\bm{1}}}

\def\va{{\bm{a}}}
\def\vb{{\bm{b}}}
\def\vc{{\bm{c}}}
\def\vd{{\bm{d}}}
\def\ve{{\bm{e}}}

\def\vp{{\bm{p}}}

\def\vs{{\bm{s}}}

\def\vu{{\bm{u}}}
\def\vv{{\bm{v}}}
\def\vw{{\bm{w}}}
\def\vx{{\bm{x}}}
\def\vy{{\bm{y}}}
\def\vz{{\bm{z}}}

\def\evb{{b}}
\def\evc{{c}}
\def\evd{{d}}

\def\evp{{p}}
\def\evq{{q}}

\def\evx{{x}}
\def\evy{{y}}

\def\mA{{\bm{A}}}

\def\mD{{\bm{D}}}
\def\mE{{\bm{E}}}
\def\mF{{\bm{F}}}

\def\mH{{\bm{H}}}
\def\mI{{\bm{I}}}
\def\mJ{{\bm{J}}}

\def\mLambda{{\bm{\Lambda}}}
\def\mSigma{{\bm{\Sigma}}}

\DeclareMathAlphabet{\mathsfit}{\encodingdefault}{\sfdefault}{m}{sl}
\SetMathAlphabet{\mathsfit}{bold}{\encodingdefault}{\sfdefault}{bx}{n}

\def\gG{{\mathcal{G}}}

\def\sA{{\mathbb{A}}}

\def\sK{{\mathbb{K}}}

\def\sN{{\mathbb{N}}}

\def\sR{{\mathbb{R}}}

\def\sW{{\mathbb{W}}}
\def\sX{{\mathbb{X}}}
\def\sY{{\mathbb{Y}}}
\def\sZ{{\mathbb{Z}}}

\def\emLambda{{\Lambda}}

\def\emSigma{{\Sigma}}

\DeclareMathOperator*{\argmax}{arg\,max}
\DeclareMathOperator*{\argmin}{arg\,min}